\documentclass[reqno,12pt,a4paper]{amsart}

\UseRawInputEncoding
\usepackage{amssymb}
\usepackage{bbm}

\usepackage{tikz}
\usetikzlibrary{matrix,arrows,calc,snakes,patterns,decorations.markings}

\usepackage{booktabs}
\usepackage[mathscr]{eucal}
\usepackage{graphics,epic}
\usepackage{amsfonts}
\usepackage{amscd}
\usepackage{latexsym}
\usepackage{amsmath,amssymb, amsthm, stmaryrd, bm, bbm}
\usepackage[all,2cell]{xy}
\usepackage{mathrsfs}
\usepackage{color}

\newtheorem{theorem}[subsection]{Theorem}
\newtheorem*{theorem*}{Theorem}

\newtheorem{lemma}[subsection]{Lemma}
\newtheorem{proposition}[subsection]{Proposition}
\newtheorem{corollary}[subsection]{Corollary}

\newtheorem*{conjecture*}{Conjecture}

\newtheorem*{question*}{Question}
\theoremstyle{remark}
\newtheorem{remark}[subsection]{Remark}
\newtheorem{example}[subsection]{Example}
\theoremstyle{definition}
\newtheorem{definition}[subsection]{Definition}
\newtheorem*{notation*}{Notation}

\newcommand{\opname}[1]{\operatorname{\mathsf{#1}}}

\renewcommand{\mod}{\opname{mod}\nolimits}

\newcommand{\add}{\opname{add}\nolimits}

\newcommand{\rad}{\opname{rad}\nolimits}

\newcommand{\ind}{\opname{ind}}

\newcommand{\Hom}{\opname{Hom}}
\newcommand{\End}{\opname{End}}

\newcommand{\Ext}{\opname{Ext}}

\newcommand{\ten}{\otimes}

\newcommand{\tilt}{\opname{tilt}}
\newcommand{\tors}{\opname{tors}}

\numberwithin{equation}{section}

\begin{document}

\title[]{Gluing support $\tau$-tilting modules via symmetric ladders of height 2}

\author{Yingying Zhang}

\address{Department of Mathematics, Huzhou University, Huzhou 313000, Zhejiang Province, P.R.China}

\email{yyzhang@zjhu.edu.cn}

\begin{abstract}
Gluing techniques with respect to a recollement have long been studied. Recently, ladders of recollements of abelian categories were introduced as important tools. We present explicit constructions of gluing of support $\tau$-tilting modules via symmetric ladders of height two. Moreover, we apply the result to triangular matrix algebras to give a more detailed version of the known Jasso's reduction and study maximal green sequences.
\end{abstract}

 \maketitle

\section{Introduction}\label{s:introduction}
\medskip

In representation theory of finite-dimensional algebras, the notion of tilting modules is a fundamental tool that aims at comparing the representation theory of an algebra with that of the endomorphism algebra of a tilting module over that algebra. Also tilting theory is a rich source of derived equivalences. As a generalization of tilting modules, Adachi-Iyama-Reiten introduced the notion of support $\tau$-tilting modules by using the Auslander-Reiten translation $\tau$ which completes tilting theory from the viewpoint of mutation [AIR].

Adachi-Iyama-Reiten established in [AIR] a bijection between functorially finite torsion classes and support $\tau$-tilting modules. This bijection will be central in our approach. Recently, Asai investigated semibricks from the point of view of $\tau$-tilting theory and showed that left finite semibricks also correspond bijectively to functorially finite torsion classes and support $\tau$-tilting modules [As]. Moreover, support $\tau$-tilting modules also have close relation to some other important notions, such as silting complexes, cluster-tilting objects, wide subcategories and maximal green sequences, see [As, AIR, BST, DIJ, DIRRT]. Support $\tau$-tilting modules play, thus, a very important role.

A recollement of abelian categories is a short exact sequence of abelian categories where both the inclusion and the quotient functors admit left and right adjoints. They first appeared in the construction of the category of perverse sheaves on a singular space by Beilinson, Bernstein and Deligne [BBD], arising from recollements of triangulated categories. Gluing techniques with respect to a recollement, due to Beilinson, Bernstein and Deligne have been intensively studied in [BBD] for t-structures and simple modules. Gluing co-t-structures was studied in [B]. Liu-Vit\'oria-Yang discussed gluing of silting objects via a recollement of bounded derived categories of finite-dimensional algebras over a field with respect to the gluing of co-t-structures [LVY], recently, in [SZ] with respect to the gluing of t-structures. Gluing torsion classes was studied by Liu-Vit\'oria-Yang in [LVY] and by Ma-Huang in [MH]. This leads to the natural question of how to glue support $\tau$-tilting modules and which support $\tau$-tilting modules are glued from smaller ones. Indeed, recent work by Zhang has given explicit constructions of gluing semibricks similar to gluing simple modules and then showed that, in the $\tau$-tilting finite case, we can construct support $\tau$-tilting modules by gluing semibricks via recollements [Zh]. However, the construction of support $\tau$-tilting modules there is less explicit. It turns out that an answer to the problem of gluing support $\tau$-tilting modules can be given more explicitly when the focus is on functorially finite torsion classes in symmetric ladders of height 2 rather than semibricks in recollements. Our first main result is as follows.
\begin{theorem}{\rm(Theorem 4.4)}
Let $A, B, C$ be finite-dimensional algebras over a field and $\mathcal {R}$ a symmetric ladder of $\mod A$ relative to $\mod B$ and $\mod C$ of height $2$ of the form
\[
\xymatrix@C=5pc{\mod B
\ar@<+1ex>[r]|{i_{0}} \ar@<-3ex>[r]|{i_{-2}} &\mod A \ar@<1ex>[l]|{i_{-1}}
\ar@<-3ex>[l]|{i_{1}}
 \ar@<+1ex>[r]|{j_{0}} \ar@<-3ex>[r]|{j_{-2}} & \mod C.\ar@<1ex>[l]|{j_{-1}} \ar@<-3ex>[l]|{j_{1}}}
\]
If $X$ and $Y$ are respectively support $\tau$-tilting modules in $\opname{mod}B$ and $\opname{mod}C$, with the corresponding functorially finite torsion classes  $\mathcal{T}_{X}$ and $\mathcal{T}_{Y}$, then the induced functorially finite torsion class $\mathcal{T}_{Z}=\{Z\in \mod A|j_{0}(Z)\in\mathcal {T}_{Y}, i_{-1}(Z)\in\mathcal {T}_{X}\}$ in $\mod A$ is associated with the support $\tau$-tilting module $Z=i_{0}(X)\oplus K_{Y}$, where $K_{Y}$ is defined by the minimal left $\mathcal{T}_{Z}$-approximation $f: j_{1}(Y)\longrightarrow K_{Y}$.
\end{theorem}
We refer to the process from $X$ and $Y$ to $Z$ in Theorem 1.1 as \emph{gluing support $\tau$-tilting modules via symmetric ladders of height 2}. Actually by [GKP, Propositions 3.1 and 3.2] symmetric ladders of module categories of height 2 only can exist for triangular matrix algebras. Furthermore, the construction of support $\tau$-tilting  modules over triangular matrix algebras has been discussed in [GH].  We apply the above result to triangular matrix algebras and show that the construction in [GH] is a particular case of the following construction. Let $B$ and $C$ be finite-dimensional algebras over a field and $M$ a finitely generated $C$-$B$-bimodule. Let $A=\left(\begin{array}{cc} B & 0\\ {}_CM_B & C\end{array}\right)$ be the triangular matrix algebra. The following is our main theorem concerning triangular matrix algebras.
\begin{theorem}{\rm(Theorem 5.1)}
If $X$ and $Y$ are respectively support $\tau$-tilting modules in $\opname{mod}B$ and $\opname{mod}C$, then $(X, 0)\oplus (X_{Y}, Y)_{f}$ is a support $\tau$-tilting $A$-module, where $X_{Y}$ and $f$ are defined by the minimal left $\opname{Fac}X$-approximation $f: Y\otimes _{C}M\longrightarrow X_{Y}$.
\end{theorem}

Here we give a more detailed version of the known Jasso's reduction in the setting of triangular matrix algebras.
\begin{theorem}{\rm(Theorem 5.6)}
If $X$ is a $\tau$-tilting $B$-module, we have an order-preserving bijection between the set of isomorphism classes of basic support $\tau$-tilting $A$-modules which have $(X, 0)$ as a direct summand and the set of isomorphism classes of all basic support $\tau$-tilting $C$-modules as follows
$$\mathrm{red}: \opname{s\tau-tilt\,}_{(X, 0)}A
\longleftrightarrow \opname{s\tau-tilt\,}C
$$
given by $Z=(R, S)_{f} \mapsto S$ with inverse $Y \mapsto (X, 0)\oplus (X_{Y}, Y)_{f}$ where $X_{Y}$ and $f$ are defined by the minimal left $\opname{Fac}X$-approximation $f: Y\otimes _{C}M\longrightarrow X_{Y}$. In particular, $\opname{s\tau-tilt\,}C$ can be embedded as an interval in $\opname{s\tau-tilt\,}A$.
\end{theorem}

Maximal green sequences were introduced by Keller in [Ke] to provide explicit formulas for the refined Donaldson-Thomas invariants of Kontsevich and Soibelman. Br\"ustle, Smith and Treffinger adapted the definition in module categories and provided a characterization of maximal green sequences in terms of $\tau$-tilting theory [BST]. Moreover, maximal green sequences of a finite dimensional algebra are in bijection with sequences of bricks which is called complete forward hom-orthogonal sequences [KD, Theorem 5.4]. Br\"ustle, Dupont and Perotin in [BDP] presented the ``no gap conjecture" for lengths of maximal green sequences and hence a number of algebraists started to study lengths of maximal green sequences for special algebras, see [BDP, GM, HI, KN]. Here we give the construction of maximal green sequences over triangular matrix algebras as follows.
\begin{theorem}{\rm(Theorem 5.10)}
If $\alpha: 0=\mathscr{X}_{0}\subsetneq \mathscr{X}_{1}\subsetneq \cdot \cdot \cdot \subsetneq \mathscr{X}_{r-1}\subsetneq \mathscr{X}_{r}=\mod B$ and $\beta: 0=\mathscr{Y}_{0}\subsetneq \mathscr{Y}_{1}\subsetneq \cdot \cdot \cdot \subsetneq \mathscr{Y}_{s-1}\subsetneq \mathscr{Y}_{s}=\mod C$ are respectively maximal green sequences in $\mod B$ and $\mod C$, then $\gamma: 0=\mathscr{Z}_{0}\subsetneq \mathscr{Z}_{1}\subsetneq \cdot \cdot \cdot \mathscr{Z}_{r}\subsetneq \mathscr{Z}_{r+1} \cdot \cdot \cdot\subsetneq \mathscr{Z}_{r+s-1}\subsetneq \mathscr{Z}_{r+s}=\mod A$ is a maximal green sequence in $\mod A$, where $\mathscr{Z}_{i}=\{(X,0)|X\in \mathscr{X}_{i}\}$ for $0\leq i\leq r$ and $\mathscr{Z}_{r+j}=\{(X,Y)_{f}|X\in \mod B, Y\in\mathscr{Y}_{j} \}$ for $1\leq j\leq s$. In particular, the lengths of $\alpha$ and $\beta$ sum up to that of $\gamma$.
\end{theorem}

The rest of this paper is organized as follows.

In Section 2, we recall some preliminary results on support $\tau$-tilting modules, torsion classes, semibricks, maximal green sequences, recollements and ladders needed for the later sections.

In Section 3, we mainly investigate the relation of covariantly finite subcategories (resp. contravariantly finite subcategories) in the middle categories of recollements and in the edge. Then we apply the results to symmetric ladders of abelian categories of height 2 and show that gluing torsion classes can be restricted to functorially finiteness condition.

In Section 4, we show how to glue support $\tau$-tilting modules by gluing functorially finite torsion classes via symmetric ladders of height 2 and gluing semibricks can be restricted to left finiteness condition.

In Section 5, we apply the result to triangular matrix algebras to give the gluing and reduction of support $\tau$-tilting modules. Moreover, we study maximal green sequences.

\section{Preliminaries}
\medskip

In this section, we collect some basic materials that will be used later. This concerns the relationship among functorially finite torsion classes, left finite semibricks, maximal green sequences and support $\tau$-tilting modules. Further, we recall known results on recollements of abelian categories and the definition of symmetric ladders of abelian categories.

Let $k$ be a field and $A$ a finite-dimensional $k$-algebra. We denote by $\opname{mod}A$ (resp. $\opname{proj}A$) the category of finitely generated (resp. finitely generated projective) right $A$-modules and by $\tau$ the Auslander-Reiten translation of $A$. The composition of maps or functors $f:X\longrightarrow Y$ and $g:Y\longrightarrow Z$ is denoted by $gf$. For $X\in\mod A$, we denote by $\opname{add}X$ (resp. $\opname{Fac}X$) the category of all direct summands (resp. factor modules) of finite direct sums of copies of $X$ and $|X|$ denotes the number of non-isomorphic indecomposable direct summands of $X$.
\vspace{0.2cm}

{\bf 2.1. $\tau$-tilting theory}
\vspace{0.2cm}

 We first recall the definition of support $\tau$-tilting modules from [AIR].
\begin{definition}
Let ($X$,$P$) be a pair with $X \in \mod A$ and $P \in \opname{proj}A$.
\begin{itemize}
\item[(1)] We call $X$ in $\mod A$ $\tau$-\emph{rigid} if  ${\opname{Hom}}_{A}(X, \tau X)$=0. We call ($X$,$P$) a $\tau$-\emph{rigid pair} if $X$ is $\tau$-rigid and ${\opname{Hom}}_{A}(P, X)$=0.
\item[(2)] We call $X$ in $\mod A$ $\tau$-\emph{tilting} if $X$ is $\tau$-rigid and $ |X| = |A|$.
\item[(3)] We call $X$ in $\mod A$ \emph{support $\tau$-tilting} if there exists an idempotent $e$ of $A$ such that $X$ is a $\tau$-tilting ($A/\langle e\rangle$)-module. We call ($X$,$P$) a \emph{support $\tau$-tilting pair} if ($X$,$P$) is $\tau$-rigid and $|X|+|P|=|A|$.
\end{itemize}
\end{definition}
An $A$-module $X$ is \emph{basic} if in a direct sum decomposition into indecomposable modules, no indecomposable module appears more than once, see [AR, Page 112]. We say that ($X$,$P$) is \emph{basic} if $X$ and $P$ are basic. Moreover, $X$ determines $P$ uniquely up to equivalence. We denote by $\opname{s\tau-tilt\,}A$ (resp. $\opname{\tau-tilt\,}A$) the set of isomorphism classes of basic support $\tau$-tilting $A$-modules (resp. $\tau$-tilting $A$-modules).
\vspace{0.2cm}

Let $\mathcal {A}$ be an abelian category and $\mathcal{T}$ a full subcategory of $\mathcal {A}$.
Assume that $T\in \mathcal{T}$ and $D\in\mathcal {A}$. The morphism $f: D\to T$ is called a \emph{left $\mathcal{T}$-approximation} of $D$ if
$$\Hom_{\mathcal {A}}(T, T')\to \Hom_{\mathcal {A}}(D, T')\to 0$$
is exact for any $T'\in\mathcal{T}$.
The subcategory $\mathcal{T}$ is called \emph{covariantly finite} in $\mathcal {A}$ if
every module in $\mathcal {A}$ has a left $\mathcal{T}$-approximation. The notions of \emph{right
$\mathcal{T}$-approximations} and \emph{contravariantly finite subcategories} of $\mathcal {A}$ can be defined dually. The
subcategory $\mathcal{T}$ is called \emph{functorially finite} in $\mathcal {A}$ if it is both covariantly finite and contravariantly
finite in $\mathcal {A}$ ([AR]).

\vspace{0.2cm}

{\bf 2.2. Functorially finite torsion class}

\vspace{0.2cm}

Let $\mathcal {A}$ be an abelian category. We say that a full subcategory $\mathcal{T}$ of $\mathcal {A}$ is a \emph{torsion class} if it is closed under images, infinite sums and extensions (see [D, Theorem 2.3]). Considering the module category $\mod A$, we denote by $\opname{f-\tors} A$ (resp. $\opname{sf-\tors} A$) the set of functorially finite (resp. sincere functorially finite) torsion classes in $\mod A$. We say that $X \in \mathcal{T}$ is Ext-projective if $\Ext$$^{1}_{A}(X,\mathcal{T})$=0. We denote by P($\mathcal{T}$) the direct sum of one copy of each of the indecomposable $\Ext$-$\opname{proj}$ objects in $\mathcal{T}$ up to isomorphism. We have P($\mathcal{T}$)$\in\mod A$ if $\mathcal{T}\in\opname{f-\tors} A$ ([AS, Corollary 4.4]). The following result gives a relationship between $\opname{s\tau-tilt\,}A$ and $\opname{f-\tors} A$.
\begin{theorem}{\rm(See [AIR, Theorem 2.7 and Corollary 2.8])}
There is a bijection
\begin{center}
$\opname{s\tau-tilt\,}A \longleftrightarrow \opname{f-\tors} A,$
\end{center}
given by $\opname{s\tau-tilt\,}A \ni T \mapsto \opname{Fac}T \in \opname{f-\tors} A$ and $\opname{f-\tors} A \ni \mathcal{T} \mapsto P(\mathcal{T}) \in\opname{s\tau-\tilt} A$, which restricts a bijection
\begin{center}
$\opname{\tau-tilt\,}A \longleftrightarrow \opname{sf-\tors} A.$
\end{center}
\end{theorem}

\vspace{0.2cm}

{\bf 2.3. Semibricks}

\vspace{0.2cm}
Recall the definition of semibricks in [As, Definition 2.1].
\begin{definition}
\begin{enumerate}
\item A module $S\in \opname{mod}A$ is called a \emph{brick} if  $\opname{End}_{A}(S)$ is a division $k$-algebra (i.e., the non-trivial endomorphisms are invertible). We write $\opname{brick}A$ for the set of isoclasses of bricks in $\opname{mod}A$.
\item A subset $\mathcal{S}\subset \opname{brick}A$ is called a \emph{semibrick} if $\opname{Hom}_{A}(S_{1}, S_{2})=0$ for any $S_{1}\neq S_{2}\in\mathcal{S}$. We write $\opname{sbrick}A$ for the set of semibricks in $\opname{mod}A$.
\end{enumerate}
\end{definition}
For example, a simple module is a brick and a set of isoclasses of simple modules is a semibrick by Schur's Lemma. Moreover, preprojective modules and preinjective modules over a finite-dimensional algebra are also bricks in the module categories (see [ASS, \uppercase\expandafter{\romannumeral8}.2.7. Lemma]).

Let $\mathcal {C}\subset \mod A$ be a full subcategory. We use the following notations.\\
$\bullet \opname{add}\mathcal {C}$ is the additive closure of $\mathcal {C}$,\\
$\bullet \opname{Fac}\mathcal {C}$ consists of the factor modules of objects in $\opname{add}\mathcal {C}$,\\
$\bullet \opname{Filt}\mathcal {C}$ consists of the objects $M$ such that there exists a sequence $0=M_{0}\subset M_{1}\subset \cdot \cdot \cdot \subset M_{n}=M$ with $M_{i}/M_{i-1}\in \opname{add}\mathcal {C}$ for all $1\leq i\leq n$.

We write $\opname{T}(\mathcal {C})$ for the smallest torsion class containing $\mathcal {C}$. It is well known that $\opname{T}(\mathcal {C})=\opname{Filt}(\opname{Fac}\mathcal {C})$ (see [MS, Lemma 3.1]). Recall the definition of left finiteness of semibricks in [As, Definition 2.2].
\begin{definition}
Let $\mathcal {S}\in \opname{sbrick}A$. The semibrick $\mathcal {S}$ is said to be \emph{left finite} if $\opname{T}(\mathcal {S})$ is functorially finite. We write $\opname{f_{L}-sbrick}A$ for the set of left finite semibricks in $\mod A$.
\end{definition}
The following result gives bijections among $\opname{f_{L}-sbrick}A$,  $\opname{f-\tors} A$ and $\opname{s\tau-tilt\,}A$.
\begin{proposition}{\rm{([As, Proposition 2.9])}}
Let $M\in \opname{s\tau-tilt\,}A$ and $B:=\opname{End}_{A}(M)$. We have the following commutative diagram of bijections:
$$
\begin{xy}
(-40,  0) *+{\opname{s\tau-tilt\,}A} = "0",
(  0,  0) *+{\opname{f-\tors} A} = "1",
( 40,  0) *+{\opname{f_{L}-sbrick}A} = "3",
(-40, -8) = "4",
( 40, -8) = "5",
\ar@<1.0mm>^{\opname{Fac}} "0"; "1"
\ar_{\opname{T}} "3"; "1"
\ar@{-} "0"; "4"
\ar@{-}_{M \mapsto \ind(M/{\rad_B M})} "4"; "5"
\ar "5"; "3"
\ar@<1.0mm>^{P} "1"; "0"
\end{xy}.
$$
\end{proposition}

\vspace{0.2cm}

{\bf 2.4. Maximal green sequences}

\vspace{0.2cm}

Recall the definition of maximal green sequences from [BST, Definition 4.8].
\begin{definition}
A \emph{maximal green sequence} in $\mod A$ is a finite sequence of torsion classes $0=\mathscr{T}_{0}\subsetneq \mathscr{T}_{1}\subsetneq \cdot \cdot \cdot \subsetneq \mathscr{T}_{r-1}\subsetneq \mathscr{T}_{r}=\mod A$ such that for all $0\leq i\leq r-1$, the inclusions $\mathscr{T}_{i}\subsetneq \mathscr{T}_{i+1}$ are covering relations (i.e., the existence of a torsion class $\mathscr{T}$ satisfying $\mathscr{T}_{i}\subseteq \mathscr{T}\subseteq \mathscr{T}_{i+1}$ implies $\mathscr{T}=\mathscr{T}_{i}$ or $\mathscr{T}=\mathscr{T}_{i+1}$).
\end{definition}
The following result gives a relation between maximal green sequence and support $\tau$-tilting modules.
\begin{proposition}{\rm{([BST, Proposition 4.9])}}
Let $0=\mathscr{T}_{0}\subsetneq \mathscr{T}_{1}\subsetneq \cdot \cdot \cdot \subsetneq \mathscr{T}_{r-1}\subsetneq \mathscr{T}_{r}=\mod A$ be a maximal green sequence in $\mod A$. Then there exists a set of $\{M_{i}\}_{i=0}^{r}$ of support $\tau$-tilting modules such that $\opname{Fac}M_{i}=\mathscr{T}_{i}$ for all $0\leq i\leq r$.
\end{proposition}

\vspace{0.2cm}

{\bf 2.5. Recollements}

\vspace{0.2cm}

For the convenience, we recall the definition of recollements of abelian categories, see for instance [BBD, FP].
\begin{definition} Let $\mathcal {A}, \mathcal {B}, \mathcal {C} $ be abelian categories. Then a recollement of $\mathcal {A}$ relative to $\mathcal {B}$ and $\mathcal {C}$, is diagrammatically expressed by
$$\xymatrix@!C=2pc{
\mathcal {B}\; \ar@{>->}[rr]|{i_{*}} && \mathcal {A} \ar@<-4.0mm>@{->>}[ll]_{i^{*}} \ar@{->>}[rr]|{j^{*}} \ar@{->>}@<4.0mm>[ll]^{i^{!}}&& \mathcal {C} \ar@{>->}@<-4.0mm>[ll]_{j_{!}} \ar@{>->}@<4.0mm>[ll]^{j_{*}}
}$$
which satisfies the following three conditions:
\begin{enumerate}
\item ($i^{*}, i_{*}$), ($i_{*}, i^{!}$), ($j_{!}, j^{*}$) and ($j^{*}, j_{*}$) are adjoint pairs;
\item $i_{*}, j_{!}$ and $j_{*}$ are fully faithful functors;
\item ${\rm Im}i_{*}={\rm Ker}j^{*}$.
\end{enumerate}
\end{definition}
\begin{remark}
(1) From Definition 2.8(1), it follows that $i_{*}$ and $j^{*}$ are both right adjoint functors and left adjoint functors, therefore they are exact functors of abelian categories.

(2) By the definition of recollements, we have $i^{*}i_{*}\cong id, i^{!}i_{*}\cong id, j^{*}j_{!}\cong id$ and $j^{*}j_{*}\cong id$. Also $i^{*}j_{!}=0, i^{!}j_{*}=0$.

(3) Throughout this paper, we denote by $R(\mathcal {B}, \mathcal {A}, \mathcal {C})$ (resp. $R(B, A, C)$) a recollement of $\mathcal {A}$ (resp. $\opname{mod}A$) relative to $\mathcal {B}$ (resp. $\opname{mod}B$) and $\mathcal {C}$ (resp. $\opname{mod}C$) as above.
\end{remark}
The definition of ladders of recollements of abelian categories is given in [GKP, Definition 1.1]. Unlike ladders of triangulated categories [AKLY], ladders of abelian categories might be asymmetry for dealing with more general cases. We only use symmetric ladders in this paper.
A \emph{symmetric ladder} $\mathcal {R}$ is a finite or infinite
diagram of abelian categories and functors
\vspace{-10pt}\[{\setlength{\unitlength}{0.7pt}
\begin{picture}(200,170)
\put(0,70){$\xymatrix@!=3pc{\mathcal {B}
\ar@<+2.5ex>[r]|{i_{n+1}}\ar@<-2.5ex>[r]|{i_{n-1}} &\mathcal{A}
\ar[l]|{j_n} \ar@<-5.0ex>[l]|{j_{n+2}} \ar@<+5.0ex>[l]|{j_{n-2}}
\ar@<+2.5ex>[r]|{j_{n+1}} \ar@<-2.5ex>[r]|{j_{n-1}} &
\mathcal{C}\ar[l]|{i_n} \ar@<-5.0ex>[l]|{i_{n+2}}
\ar@<+5.0ex>[l]|{i_{n-2}}}$} 
\put(52.5,10){$\vdots$} 
\put(137.5,10){$\vdots$} 
\put(52.5,130){$\vdots$} 
\put(137.5,130){$\vdots$} 
\end{picture}}
\]
such that any three consecutive rows form a recollement (see [BGS, FZ, GKP]). The \emph{height} of a symmetric ladder is the number of recollements
contained in it (counted with multiplicities). A recollement
is considered to be a ladder of height 1. Actually, Feng and Zhang's classification states that the height is an element of $\{0, 1, 2, \infty\}$ ([FZ, Theorem 1.1]). In this paper, we mainly concern symmetric ladders of abelian categories of height 2 which contain two recollements called upper and lower recollements.

\section{Gluing functorially finite torsion classes}
\medskip

In this section, we mainly investigate covariantly finite subcategories (resp. contravariantly finite subcategories) in recollements of abelian categories. The version in recollements of triangulated categories was already discussed in [T]. Then we apply the result to symmetric ladders of abelian categories of height 2. The following proposition is very useful in the sequel.
\begin{proposition}
Let $R(\mathcal {B}, \mathcal {A}, \mathcal {C})$ be a recollement of abelian categories. Assume that $\mathcal {T}_{\mathcal {B}}$ and $\mathcal {T}_{\mathcal {C}}$ are subcategories of $\mathcal {B}$ and $\mathcal {C}$, $\mathcal {T}_{\mathcal {A}}=\{T_{A}\in \mathcal {A}|j^{*}(T_{A})\in\mathcal {T}_{\mathcal {C}}, i^{!}(T_{A})\in\mathcal {T}_{\mathcal {B}}\}$ and $\mathcal {M}_{\mathcal {A}}=\{M_{A}\in \mathcal {A}|j^{*}(M_{A})\in\mathcal {T}_{\mathcal {C}}, i^{*}(M_{A})\in\mathcal {T}_{\mathcal {B}}\}$ are subcategories of $\mathcal {A}$. Then we have
\begin{enumerate}
\item If $\mathcal {A}$ has enough projective objects and $i^{!}$ is exact, then $\mathcal {T}_{\mathcal {B}}$ and $\mathcal {T}_{\mathcal {C}}$ are covariantly finite if and only if $\mathcal {T}_{\mathcal {A}}$ is covariantly finite;
\item If $\mathcal {A}$ has enough injective objects and $i^{*}$ is exact, then $\mathcal {T}_{\mathcal {B}}$ and $\mathcal {T}_{\mathcal {C}}$ are contravariantly finite if and only if $\mathcal {M}_{\mathcal {A}}$ is contravariantly finite;
\item If $\mathcal {A}$ has enough projective and injective objects, $i^{!}$ and $i^{*}$ are exact, then $\mathcal {T}_{\mathcal {B}}$ and $\mathcal {T}_{\mathcal {C}}$ are functorially finite if and only if $\mathcal {T}_{\mathcal {A}}(=\mathcal {M}_{\mathcal {A}})$ is functorially finite.
\end{enumerate}
\end{proposition}
\begin{proof}
We only prove (1) since the proofs of (1) and (2) are dual to each other. By [MH, Lemma 2(3)] or [FZ, Proposition 3.4], if $i^{!}$ and $i^{*}$ are exact, then $i^{*}\cong i^{!}$ and $j_{!}\cong j_{*}$, and $\mathcal {A}\cong \mathcal {B}\oplus \mathcal {C}.$ Thus we can easily get (3) from (1) and (2).

We first establish the``if" part. Assume $\mathcal {T}_{\mathcal {A}}=\{T_{A}\in \mathcal {A}|j^{*}(T_{A})\in\mathcal {T}_{\mathcal {C}}, i^{!}(T_{A})\in\mathcal {T}_{\mathcal {B}}\}$ is covariantly finite in $\mathcal {A}$. For any $X\in \mathcal {B}$, its image $i_{*}(X)\in \mathcal {A}$. Since $\mathcal {T}_{\mathcal {A}}$ is covariantly finite, there is a left $\mathcal {T}_{\mathcal {A}}$-approximation $a: i_{*}(X)\longrightarrow T_{A}$. It follows that $i^{!}(T_{A})\in\mathcal {T}_{\mathcal {B}}$ since $T_{A}\in \mathcal {T}_{\mathcal {A}}.$ We claim that $i^{!}a: X\longrightarrow i^{!}(T_{A})$ is a left $\mathcal {T}_{\mathcal {B}}$-approximation of $X$.

Let $M_{B}\in \mathcal {T}_{\mathcal {B}}$ and $b: X\longrightarrow M_{B}$. It is easy to check that $i_{*}(M_{B})\in \mathcal {T}_{\mathcal {A}}$ and $i_{*}b: i_{*}(X)\longrightarrow i_{*}(M_{B})$. Since $a$ is a left $\mathcal {T}_{\mathcal {A}}$-approximation, there exists $c$ such that $i_{*}b=ca$. Thus we have $b\cong i^{!}i_{*}b=(i^{!}c)(i^{!}a)$. This shows that $\mathcal {T}_{\mathcal {B}}$ is covariantly finite. The proof that $\mathcal {T}_{\mathcal {C}}$ is covariantly finite when $\mathcal {T}_{\mathcal {A}}$ is covariantly finite is similar.

Assume now that $\mathcal {T}_{\mathcal {B}}$ and $\mathcal {T}_{\mathcal {C}}$ are covariantly finite. By [T, Lemma 2] or [Z, Theorem 2.1], it follows that $j_{*}(\mathcal {T}_{\mathcal {C}})$ and $i_{*}(\mathcal {T}_{\mathcal {B}})$ are covariantly finite in $\mathcal {A}$.

For any $X\in \mathcal {A}$, let $f: X\longrightarrow j_{*}(T_{C})$ be the left $j_{*}(\mathcal {T}_{\mathcal {C}})$-approximation of $X$, where $T_{C}\in \mathcal {T}_{\mathcal {C}}$. Since $\mathcal {A}$ has enough projective objects, we consider the epimorphism $\pi :P\longrightarrow j_{*}(T_{C})$, where $P$ is the projective cover of $j_{*}(T_{C})$. Consider the short exact sequence $0\longrightarrow K\longrightarrow X\oplus P\longrightarrow j_{*}(T_{C})\longrightarrow 0$. Let $h: K\longrightarrow i_{*}(T_{B})$ be the left $i_{*}(\mathcal {T}_{\mathcal {B}})$-approximation of $K$. Consider the pushout diagram:
$$\xymatrix{
& 0\ar[r] & K\ar[r]^{g}\ar[d]_{h} & X\oplus P\ar[r]^{(f,\pi)}\ar[d]^{(l, l')} & j_{*}(T_{C})\ar[r]\ar@{=}[d] & 0&\\
& 0\ar[r] &i_{*}(T_{B})\ar[r]^{k} & T\ar[r] & j_{*}(T_{C})\ar[r] & 0.&}$$
Applying $i^{!}$ to the second row of the above diagram, by Remark 2.9(2) we have $i^{!}(T)\cong T_{B}\in \mathcal {T}_{\mathcal {B}}$. Applying $j^{*}$ to the second row of the above diagram, by Remark 2.9(2) we have $j^{*}(T)\cong T_{C}\in \mathcal {T}_{\mathcal {C}}$. Thus $T\in \mathcal {T}_{\mathcal {A}}$.

We claim that the map $l: X\longrightarrow T$ is the left $\mathcal {T}_{\mathcal {A}}$-approximation of $X$.

Let $M\in \mathcal {T}_{\mathcal {A}}$ and $m: X\longrightarrow M$. Since $i^{!}$ is exact, by [FZ, Lemma 3.1(4)] we have an exact sequence $0\rightarrow i_{*}i^{!}(M)\buildrel {\sigma_{M}} \over\rightarrow M\buildrel {\varepsilon_{M}} \over\rightarrow j_{*}j^{*}(M)\rightarrow 0$ and a map $\varepsilon_{M}\cdot m: X\longrightarrow j_{*}j^{*}(M)$. Since $j^{*}(M)\in \mathcal {T}_{\mathcal {C}}$ and $f: X\longrightarrow j_{*}(T_{C})$ is the left $j_{*}(\mathcal {T}_{\mathcal {C}})$-approximation of $X$, there exists $p: j_{*}(T_{C})\longrightarrow j_{*}j^{*}(M)$ such that $\varepsilon_{M}\cdot m=pf$.

Since $P$ is projective and $\varepsilon_{M}$ is an epimorphism, there exists $m': P\longrightarrow M$ such that $p\pi=\varepsilon_{M}\cdot m'$. Hence we have the following commutative diagram:
$$\xymatrix{
& 0\ar[r] & K\ar[r]^{g}\ar@{.>}[d]_{q} & X\oplus P\ar[r]^{(f,\pi)}\ar[d]^{(m,m')} & j_{*}(T_{C})\ar[r]\ar[d]^{p} & 0&\\
& 0\ar[r] &i_{*}i^{!}(M)\ar[r]^{\sigma_{M}} & M\ar[r]^{\varepsilon_{M}} & j_{*}j^{*}(M)\ar[r] & 0.&}$$

Note that $i^{!}(M)\in \mathcal {T}_{\mathcal {B}}$ and $h: K\longrightarrow i_{*}(T_{B})$ is the left $i_{*}(\mathcal {T}_{\mathcal {B}})$-approximation of $K$, thus there exists $r: i_{*}(T_{B})\longrightarrow i_{*}i^{!}(M)$ such that $q=rh$. Then it follows from the left square of the above diagram that $(m, m', -\sigma_{M}\cdot r)\left(
 \begin{array}{ccc}
 g\\
 h\\
 \end{array}
\right)=(m, m')g-\sigma_{M}\cdot rh=\sigma_{M}\cdot q-\sigma_{M}\cdot rh=0$. From the pushout diagram we have an exact sequence $$0\longrightarrow K\buildrel {\left(
 \begin{array}{ccc}
 g\\
 h\\
 \end{array}
\right)} \over\longrightarrow X\oplus P\oplus i_{*}(T_{B})\buildrel {(l, l', k)} \over\longrightarrow T\longrightarrow 0,$$ thus there exists $s: T\longrightarrow M$ such that $(m, m', -\sigma_{M}\cdot r)=s(l, l', k)$. Therefore $m=sl$. We have finished to prove that $l: X\longrightarrow T$ is the left $\mathcal {T}_{\mathcal {A}}$-approximation of $X$. It follows that $\mathcal {T}_{\mathcal {A}}$ is covariantly finite in $\mathcal {A}$.
\end{proof}

We apply the above result to a symmetric ladder of abelian categories of height 2 and get the following
\begin{corollary}
Let $\mathcal {R}$ be a symmetric ladder of height $2$
\[
\xymatrix@C=5pc{\mathcal {B}
\ar@<+1ex>[r]|{i_{0}} \ar@<-3ex>[r]|{i_{-2}} &\mathcal {A} \ar@<1ex>[l]|{i_{-1}}
\ar@<-3ex>[l]|{i_{1}}
 \ar@<+1ex>[r]|{j_{0}} \ar@<-3ex>[r]|{j_{-2}} & \mathcal {C},\ar@<1ex>[l]|{j_{-1}} \ar@<-3ex>[l]|{j_{1}}}
\]
and $\mathcal {T}_{\mathcal {B}}\subseteq \mathcal {B}$, $\mathcal {T}_{\mathcal {C}}\subseteq \mathcal {C}$, $\mathcal {T}_{\mathcal {A}}=\{T_{A}\in \mathcal {A}|j_{0}(T_{A})\in\mathcal {T}_{\mathcal {C}}, i_{-1}(T_{A})\in\mathcal {T}_{\mathcal {B}}\}\subseteq \mathcal {A}$. We have
\begin{enumerate}
\item If $\mathcal {A}$ has enough projective objects, then $\mathcal {T}_{\mathcal {B}}$ and $\mathcal {T}_{\mathcal {C}}$ are covariantly finite if and only if $\mathcal {T}_{\mathcal {A}}$ is covariantly finite;
\item If $\mathcal {A}$ has enough injective objects, then $\mathcal {T}_{\mathcal {B}}$ and $\mathcal {T}_{\mathcal {C}}$ are contravariantly finite if and only if $\mathcal {T}_{\mathcal {A}}$ is contravariantly finite;
\item If $\mathcal {A}$ has enough projective and injective objects, then $\mathcal {T}_{\mathcal {B}}$ and $\mathcal {T}_{\mathcal {C}}$ are functorially finite if and only if $\mathcal {T}_{\mathcal {A}}$ is functorially finite.
\end{enumerate}
\end{corollary}
\begin{proof}
If we apply Proposition 3.1(1) to the upper recollement, we get (1). Dually,
if we apply Proposition 3.1(2) to the lower recollement, we get (2). Combine (1) and (2), we get (3).
\end{proof}
Note that the construction of $\mathcal {T}_{\mathcal {A}}$ in Corollary 3.2 has close relation to the construction of gluing torsion classes in [MH]. We show that gluing torsion classes can be restricted to functorially finiteness condition.
\begin{theorem}
Let $\mathcal {R}$ be a symmetric ladder of height $2$
\[
\xymatrix@C=5pc{\mathcal {B}
\ar@<+1ex>[r]|{i_{0}} \ar@<-3ex>[r]|{i_{-2}} &\mathcal {A} \ar@<1ex>[l]|{i_{-1}}
\ar@<-3ex>[l]|{i_{1}}
 \ar@<+1ex>[r]|{j_{0}} \ar@<-3ex>[r]|{j_{-2}} & \mathcal {C}.\ar@<1ex>[l]|{j_{-1}} \ar@<-3ex>[l]|{j_{1}}}
\]
Suppose $\mathcal {T}_{\mathcal {B}}$ and $\mathcal {T}_{\mathcal {C}}$ are functorially finite torsion classes in $\mathcal {B}$ and $\mathcal {C}$. If $\mathcal {A}$ has enough projective objects, then $\mathcal {T}_{\mathcal {A}}=\{T_{A}\in \mathcal {A}|j_{0}(T_{A})\in\mathcal {T}_{\mathcal {C}}, i_{-1}(T_{A})\in\mathcal {T}_{\mathcal {B}}\}$ is a functorially finite torsion class in $\mathcal {A}$.
\end{theorem}
\begin{proof}
Observe the form of $\mathcal {T}_{\mathcal {A}}$ in the lower recollement, we have $\mathcal {T}_{\mathcal {A}}$ is a torsion class by [MH, Theorem 1]. By [S1] it is well known that $\mathcal {T}_{\mathcal {A}}$ is contravariantly finite since $\mathcal {T}_{\mathcal {A}}$ is a torsion class. Then by Corollary 3.2(1), it is easy to get the result.
\end{proof}
Let $k$ be a field, $A, B$ and $C$ are finite-dimensional $k$ algebras. Considering ladders of module categories, we have
\begin{corollary}
Let $\mathcal {R}$ be a symmetric ladder of module categories of height $2$
\[
\xymatrix@C=5pc{\mod B
\ar@<+1ex>[r]|{i_{0}} \ar@<-3ex>[r]|{i_{-2}} &\mod A \ar@<1ex>[l]|{i_{-1}}
\ar@<-3ex>[l]|{i_{1}}
 \ar@<+1ex>[r]|{j_{0}} \ar@<-3ex>[r]|{j_{-2}} & \mod C.\ar@<1ex>[l]|{j_{-1}} \ar@<-3ex>[l]|{j_{1}}}
\]
Suppose $\mathcal {T}_{B}$ and $\mathcal {T}_{C}$ are functorially finite torsion classes in $\mod B$ and $\mod C$. Then $\mathcal {T}_{A}=\{T\in \mod A|j_{0}(T)\in\mathcal {T}_{C}, i_{-1}(T)\in\mathcal {T}_{B}\}$ is a functorially finite torsion class in $\mod A$.
\end{corollary}
Let $B$ and $C$ be finite-dimensional $k$ algebras and $M$ a finitely generated $C$-$B$-bimodule. Consider the matrix algebra $A=\left(\begin{array}{cc} B & 0\\ {}_CM_B & C\end{array}\right)$, see [ARS, Pages 73-75]. A right $A$-module is identified with a triple $(X, Y)_{f}$, or simply $(X, Y)$ if $f$ is clear, where $X$ is a right $B$-module, $Y$ is a right $C$-module, and $f: Y\otimes _{C}M\longrightarrow X$ is a right $B$-map. The morphisms from $(X, Y)_{f}$ to $(X', Y')_{f'}$ are pairs of $(\alpha, \beta)$ such that the following diagram
$$\xymatrix{
& Y\otimes _{C}M\ar[rr]^{f}\ar[d]_{\beta \otimes id_M} && X\ar[d]^{\alpha}&\\
& Y'\otimes _{C}M\ar[rr]^{f'} && X'&}$$
commutes, where $\alpha\in \Hom_{B}(X, X')$ and $\beta\in \Hom_{C}(Y, Y')$. Moreover, a sequence $$0\longrightarrow (X_{1}, Y_{1})\buildrel{(\alpha_{1}, \beta_{1})}\over\longrightarrow(X_{2}, Y_{2})\buildrel{(\alpha_{2}, \beta_{2})}\over\longrightarrow (X_{3}, Y_{3})\longrightarrow0$$ in $\mod A$ is exact if and only if $0\longrightarrow X_{1}\buildrel{\alpha_{1}}\over\longrightarrow X_{2}\buildrel{\alpha_{2}}\over\longrightarrow X_{3}\longrightarrow0$ is exact in $\mod B$ and $0\longrightarrow Y_{1}\buildrel{\beta_{1}}\over\longrightarrow Y_{2}\buildrel{\beta_{2}}\over\longrightarrow Y_{3}\longrightarrow0$ is exact in $\mod C$.

Put $e_1=\left(\begin{array}{cc} 1 & 0\\ 0 & 0\end{array}\right)$ and $e_2=\left(\begin{array}{cc} 0 & 0\\ 0 & 1\end{array}\right)$. They produce a symmetric ladder of height $2$ (see [AKLY, Example 3.4], [FZ, Section 4.5], [H, Example 1], [PV, Example 2.10])
\begin{align*}
\xymatrix@C=5pc{\mod B
\ar@<+1ex>[r]|{inc} \ar@<-3ex>[r] &\mod A \ar@<1ex>[l]|{?\ten_A Ae_1}
\ar@<-3ex>[l]
 \ar@<+1ex>[r]|{?\ten_A Ae_2} \ar@<-3ex>[r] & \mod
C.\ar@<1ex>[l]|{inc} \ar@<-3ex>[l]|{?\ten_C e_2A}}
\tag{$\mathcal {L}$}\end{align*}
\vspace{0.2cm}

As an application of Corollary 3.2, we deduce Smal$\o$'s result.
\begin{corollary}{\rm{([S2, Theorem 2.1])}}
Let $B, C, _CM_B $, and $A$ be as above and let $\mathcal {T}_{B}\subset \mod B$, $\mathcal {T}_{C}\subset \mod C$. Take $\mathcal {T}_{A}$ the full subcategory of $\mod A$ consisting of the modules of the form $(X, Y)_{f}$ where $X\in \mathcal {T}_{B}$ and $Y\in \mathcal {T}_{C}$.
\begin{enumerate}
\item $\mathcal {T}_{B}$ and $\mathcal {T}_{C}$ are covariantly finite if and only if $\mathcal {T}_{A}$ is covariantly finite;
\item $\mathcal {T}_{B}$ and $\mathcal {T}_{C}$ are contravariantly finite if and only if $\mathcal {T}_{A}$ is contravariantly finite;
\item $\mathcal {T}_{B}$ and $\mathcal {T}_{C}$ are functorially finite if and only if $\mathcal {T}_{A}$ is functorially finite.
\end{enumerate}
In particular, if $\mathcal {T}_{B}$ and $\mathcal {T}_{C}$ are functorially finite torsion classes, then $\mathcal {T}_{A}$ is a functorially finite torsion class.
\end{corollary}
\begin{proof}
By Corollary 3.2 and Corollary 3.4, we only have to prove that $$\mathcal {T}_{A}=\{T\in \mod A|T\ten_A Ae_2\in \mathcal {T}_{\mathcal {C}}, T\ten_A Ae_1\in \mathcal {T}_{\mathcal {B}}\}.$$ Assume that $T=(X, Y)_{f}$. Then $T\ten_A Ae_2\cong Te_{2}\cong Y$ and $T\ten_A Ae_1\cong Te_{1}\cong X$. The result is immediate.
\end{proof}

\section{Gluing support $\tau$-tilting modules}
\medskip
Throughout this section, let $A, B$ and $C$ be finite-dimensional algebras over a field and $\mathcal {R}$ a symmetric ladder of $\mod A$ relative to $\mod B$ and $\mod C$ of height $2$ as follows
\vspace{0.2cm}
\[
\xymatrix@C=5pc{\mod B
\ar@<+1ex>[r]|{i_{0}} \ar@<-3ex>[r]|{i_{-2}} &\mod A \ar@<1ex>[l]|{i_{-1}}
\ar@<-3ex>[l]|{i_{1}}
 \ar@<+1ex>[r]|{j_{0}} \ar@<-3ex>[r]|{j_{-2}} & \mod C.\ar@<1ex>[l]|{j_{-1}} \ar@<-3ex>[l]|{j_{1}}}
\]
\vspace{0.2cm}

By [GKP, Propositions 3.1 and 3.2] symmetric ladders of module categories of height 2 only can exist for triangular matrix algebras. Therefore $A$ is exactly the matrix algebra $\left(\begin{array}{cc} B & 0\\ {}_CM_B & C\end{array}\right)$ and,  $i_{0}$ and $j_{-1}$ are inclusion functors (see ($\mathcal {L}$) in Section 3). In this section, we mainly use a symmetric ladder as a tool to construct support $\tau$-tilting modules. More precise results about support $\tau$-tilting modules over triangular matrix algebras will be discussed in the next section.

In [Zh, Theorem 0.1], we gave the gluing of semibricks with respect to a recollement. We use $``\sqcup"$ to denote disjoint union. Here we mainly concern gluing semibricks via the upper recollement of a symmetric ladder of height 2, since gluing semibricks via the upper recollement can be restricted to left finiteness condition.
\begin{proposition}
Let $\mathcal {R}$ be a symmetric ladder of $\mod A$ relative to $\mod B$ and $\mod C$ of height $2$. If $\mathcal {S}_{X}\in \opname{sbrick}B$ and $\mathcal {S}_{Y}\in \opname{sbrick}C$, then we have $\mathcal {S}_{Z}=i_{0}(\mathcal {S}_{X})\sqcup j_{-1}(\mathcal {S}_{Y})\in \opname{sbrick}A$. More precisely, $\mathcal {S}_{Z}=\{(S, 0)|S\in \mathcal {S}_{X}\}\sqcup \{(0, S')|S'\in \mathcal {S}_{Y}\}\cong \mathcal {S}_{X}\sqcup \mathcal {S}_{Y}$.
\end{proposition}
\begin{proof}
Since $i_{-1}$ is exact, considering the upper recollement, we can get the result by [Zh, Corollary 2.1].
\end{proof}
Here we can describe that gluing functorially finite torsion classes via the lower recollement that is compatible with the bijection between functorially finite torsion classes and left finite semibricks. As a result, we show that gluing semibricks can be restricted to left finiteness condition in this case. Before doing this, the following observation is useful for proving Theorem 4.3.
\begin{lemma}
If the functor $i: \mod A\longrightarrow \mod B$ is exact and $\mathcal {T}\subset \mod A$, then we have $i(\opname{Filt}(\opname{Fac}\mathcal {T}))\subset \opname{Filt}(\opname{Fac}(i(\mathcal {T})))$.
\end{lemma}
\begin{proof}
For any $X\in\opname{Filt}(\opname{Fac}\mathcal {T})$, there exists a sequence $0=X_{0}\subset X_{1}\subset \cdot \cdot \cdot \subset X_{s-1}\subset X_{s}=X$ with $X_{i}/X_{i-1}\in \add(\opname{Fac}\mathcal {T})$ for all $1\leq i\leq s$. Apply the exact functor $i$ to the above sequence, we have a sequence $0=i(X_{0})\subset i(X_{1})\subset \cdot \cdot \cdot \subset i(X_{s-1})\subset i(X_{s})=i(X)$ with $i(X_{i})/i(X_{i-1})\cong i(X_{i}/X_{i-1})\in i(\add(\opname{Fac}\mathcal {T}))\subset \add(\opname{Fac}(i(\mathcal {T})))$ for all $1\leq i\leq s$. Therefore $i(X)\in \opname{Filt}(\opname{Fac}(i(\mathcal {T})))$.
\end{proof}
\begin{theorem}
Let $\mathcal {R}$ be a symmetric ladder of $\mod A$ relative to $\mod B$ and $\mod C$ of height $2$.
If $\mathcal{T}_{X}$ and $\mathcal{T}_{Y}$ are respectively functorially finite torsion classes in $\opname{mod}B$ and $\opname{mod}C$, with the corresponding left finite semibricks $\mathcal {S}_{X}$ and $\mathcal {S}_{Y}$, then the glued semibrick $\mathcal {S}_{Z}=i_{0}(\mathcal {S}_{X})\sqcup j_{-1}(\mathcal {S}_{Y})$ in $\mod A$ is associated with the functorially finite torsion class $\mathcal{T}_{Z}=\{Z\in \mod A|j_{0}(Z)\in\mathcal {T}_{Y}, i_{-1}(Z)\in\mathcal {T}_{X}\}$. In particular, $\mathcal {S}_{Z}$ is a left finite semibrick.
\end{theorem}
\begin{proof}
By Proposition 2.5, we only have to prove that $T(\mathcal {S}_{Z})=\mathcal{T}_{Z}$.

We first prove $T(\mathcal {S}_{Z})\subset \mathcal{T}_{Z}$. For any $Z\in T(\mathcal {S}_{Z})=\opname{Filt}(\opname{Fac}(i_{0}(\mathcal {S}_{X})\sqcup j_{-1}(\mathcal {S}_{Y})))$,
since $i_{-1}$ is exact, by Lemma 4.2 and Proposition 2.5 we have $i_{-1}(Z)\in \opname{Filt}(\opname{Fac}(\mathcal {S}_{X}))=T(\mathcal {S}_{X})=\mathcal{T}_{X}$. Similarly, since $j_{0}$ is exact, by Lemma 4.2 and Proposition 2.5 we have $j_{0}(Z)\in \opname{Filt}(\opname{Fac}(\mathcal {S}_{Y}))=T(\mathcal {S}_{Y})=\mathcal{T}_{Y}$. Thus $Z\in \mathcal{T}_{Z}$.

Conversely, for any $Z\in \mathcal{T}_{Z}$, by Corollary 3.5 $Z=(X, Y)_{f}$ where $X\in \mathcal {T}_{X}$ and $Y\in \mathcal {T}_{Y}$. Then we have an exact sequence $0\rightarrow (X, 0)\rightarrow (X, Y)_{f}\rightarrow (0, Y)\rightarrow 0$, i.e., there is a sequence $0\subset (X, 0)\subset (X, Y)_{f}$. Since $X\in \mathcal{T}_{X}=T(\mathcal {S}_{X})=\opname{Filt}(\opname{Fac}(\mathcal {S}_{X}))$, there exists a sequence $0=X_{0}\subset X_{1}\subset \cdot \cdot \cdot \subset X_{m-1}\subset X_{m}=X$ with $X_{i}/X_{i-1}\in \add(\opname{Fac}(\mathcal {S}_{X}))$ for all $1\leq i\leq m$. Since $Y\in \mathcal{T}_{Y}=T(\mathcal {S}_{Y})=\opname{Filt}(\opname{Fac}(\mathcal {S}_{Y}))$, there exists a sequence $0=Y_{0}\subset Y_{1}\subset \cdot \cdot \cdot \subset Y_{n-1}\subset Y_{n}=Y$ with $Y_{j}/Y_{j-1}\in \add(\opname{Fac}(\mathcal {S}_{Y}))$ for all $1\leq j\leq n$. Then we have a sequence $0=(X_{0}, 0)\subset (X_{1}, 0)\subset \cdot \cdot \cdot \subset (X_{m-1}, 0)\subset (X_{m}, 0)=(X, 0)=(X, Y_{0})\subset (X, Y_{1})\subset \cdot \cdot \cdot \subset (X, Y_{n-1})\subset (X, Y_{n})=(X, Y)_{f}$, where $(X_{i}, 0)/(X_{i-1}, 0)\cong X_{i}/X_{i-1}\in \add(\opname{Fac}(\mathcal {S}_{X}))\subset \add(\opname{Fac}(\mathcal {S}_{X}\sqcup \mathcal {S}_{Y}))$ for all $1\leq i\leq m$, and $(X, Y_{j})/(X, Y_{j-1})\cong Y_{j}/Y_{j-1}\in \add(\opname{Fac}(\mathcal {S}_{Y}))\subset \add(\opname{Fac}(\mathcal {S}_{X}\sqcup \mathcal {S}_{Y}))$ for all $1\leq j\leq n$. Therefore $Z\in \opname{Filt}(\opname{Fac}(\mathcal {S}_{X}\sqcup \mathcal {S}_{Y}))\cong \opname{Filt}(\opname{Fac}(i_{0}(\mathcal {S}_{X})\sqcup j_{-1}(\mathcal {S}_{Y})))=T(\mathcal {S}_{Z})$.
\end{proof}

The natural question is if we can construct support $\tau$-tilting modules by gluing functorially finite torsion classes via ladders. More precisely,
\begin{question*}
Given a symmetric ladder of module categories of height 2, support $\tau$-tilting modules $X$ and $Y$ in $\mod B$ and $\mod C$, is it possible to construct a support $\tau$-tilting $A$-module corresponding to the glued functorially finite torsion class via the lower recollement?
\end{question*}
The following theorem answers this question positively.
\begin{theorem}
Let $\mathcal {R}$ be a symmetric ladder of $\mod A$ relative to $\mod B$ and $\mod C$ of height $2$.
If $X$ and $Y$ are respectively support $\tau$-tilting modules in $\opname{mod}B$ and $\opname{mod}C$, with the corresponding functorially finite torsion classes  $\mathcal{T}_{X}$ and $\mathcal{T}_{Y}$, then the induced functorially finite torsion class $\mathcal{T}_{Z}=\{Z\in \mod A|j_{0}(Z)\in\mathcal {T}_{Y}, i_{-1}(Z)\in\mathcal {T}_{X}\}$ in $\mod A$ is associated with the support $\tau$-tilting module $Z=i_{0}(X)\oplus K_{Y}$, where $K_{Y}$ is defined by the minimal left $\mathcal{T}_{Z}$-approximation $f: j_{1}(Y)\longrightarrow K_{Y}$.
\end{theorem}
\begin{proof}
By Theorems 2.2 and 3.3, we only have to prove that $Z$ is Ext-projective in $\mathcal{T}_{Z}$ and $\mathcal{T}_{Z}=\opname{Fac}Z$.

(i) First, we show that $Z$ is Ext-projective in $\mathcal{T}_{Z}$, i.e., $\Ext_{A}^{1}(Z, \mathcal{T}_{Z})=0$. It is easy to see that $i_{0}(X)\in \mathcal{T}_{Z}$ and therefore $Z\in \mathcal{T}_{Z}$. It follows that $\opname{Fac}Z\subset \mathcal{T}_{Z}$ since $\mathcal{T}_{Z}$ is a torsion class.

Note that $i_{0}$ and $i_{-1}$ are exact adjoint functors. Then we have $\Ext_{A}^{1}(i_{0}(X), \mathcal{T}_{Z})\cong\Ext_{B}^{1}(X, i_{-1}(\mathcal{T}_{Z}))=0$ since $i_{-1}(\mathcal{T}_{Z})\subset \mathcal{T}_{X}$ and $\mathcal {P}(\mathcal{T}_{X})=X$. Now we aim to prove that $\Ext_{A}^{1}(K_{Y}, \mathcal{T}_{Z})=0$. We consider the epimorphism $\pi :P\longrightarrow K_{Y}$, where $P$ is the projective cover of $K_{Y}$. Then we have a short exact sequence
\begin{equation}
0\longrightarrow K\longrightarrow j_{1}(Y)\oplus P\longrightarrow K_{Y}\longrightarrow 0.
\end{equation}
Apply the functor $\Hom_{A}(-, \mathcal {T}_{Z})$ to (4.1), we have a long exact sequence $(\divideontimes )$
$$\Hom_{A}(j_{1}(Y)\oplus P, \mathcal {T}_{Z})\longrightarrow \Hom_{A}(K, \mathcal {T}_{Z})\longrightarrow \Ext_{A}^{1}(K_{Y}, \mathcal {T}_{Z})\longrightarrow \Ext_{A}^{1}(j_{1}(Y)\oplus P, \mathcal {T}_{Z}).$$

Step1: we claim that $\Ext_{A}^{1}(j_{1}(Y)\oplus P, \mathcal{T}_{Z})=\Ext_{A}^{1}(j_{1}(Y), \mathcal{T}_{Z})=0$. We only have to prove $\Ext_{A}^{1}(j_{1}(Y), T)=0$ for any $T\in \mathcal {T}_{Z}$.

Consider the exact sequence
\begin{equation}
0\rightarrow T\buildrel {p} \over\rightarrow I\buildrel {q } \over\rightarrow L\rightarrow 0
\end{equation}
where $I$ is the injective envelope of $T$. Apply the functor $\Hom_{A}(j_{1}(Y), -)$ to (4.2), we have a long exact sequence
$$\xymatrix{
& \Hom_{A}(j_{1}(Y), I)\ar[r]^{q_{*}}\ar[d]^{\cong} & \Hom_{A}(j_{1}(Y), L)\ar[r]\ar[d]^{\cong} & \Ext_{A}^{1}(j_{1}(Y), T)\ar[r] & 0&\\
& \Hom_{C}(Y, j_{0}(I))\ar[r]^{(j_{0}q)_{*}} &\Hom_{C}(Y, j_{0}(L)). &   & &}$$
Since $(j_{1}, j_{0})$ is an adjoint pair, we only have to prove $(j_{0}q)_{*}$ is epic. Since $j_{0}$ is exact, apply $j_{0}$ to (4.2) we have the following exact sequence
\begin{equation}
0\rightarrow j_{0}(T)\buildrel {j_{0}p} \over\rightarrow j_{0}(I)\buildrel {j_{0}q } \over\rightarrow j_{0}(L)\rightarrow 0.
\end{equation}
Apply the functor $\Hom_{C}(Y, -)$ to (4.3), we have a long exact sequence
$$\Hom_{C}(Y, j_{0}(I))\buildrel {(j_{0}q)_{*}}  \over\longrightarrow  \Hom_{C}(Y, j_{0}(L))\longrightarrow \Ext_{C}^{1}(Y, j_{0}(T)).
$$
Since $T\in \mathcal {T}_{Z}$, $j_{0}(T)\in \mathcal {T}_{Y}$. It follows that $\Ext_{C}^{1}(Y, j_{0}(T))=0$ since $\mathcal {P}(\mathcal{T}_{Y})=Y$. Thus $(j_{0}q)_{*}$ is epic.

Step2: we claim that $\Hom_{A}(j_{1}(Y)\oplus P, \mathcal {T}_{Z})\longrightarrow \Hom_{A}(K, \mathcal {T}_{Z})$ is surjective. For any $T\in \mathcal {T}_{Z}$ and a map $g: K\longrightarrow T$, we consider the pushout diagram:
$$\xymatrix{
& 0\ar[r] & K\ar[r]^{i}\ar[d]_{g} & j_{1}(Y)\oplus P\ar[r]^{(f, \pi)}\ar[d]^{(h, h')}\ar@{.>}[dl]_{t'} & K_{Y}\ar[r]\ar@{=}[d]\ar@{.>}[dl]^{tl} & 0&\\
& 0\ar[r] &T\ar[r]_{r} & C\ar[r]_{s} & K_{Y}\ar[r] & 0.&}$$
Since $T$ and $K_{Y}$ are both in $\mathcal {T}_{Z}$ and $\mathcal {T}_{Z}$ is a torsion class, we have $C\in \mathcal {T}_{Z}$. Observe that $h: j_{1}(Y)\longrightarrow C$, there exists $t: K_{Y}\longrightarrow C$ such that $h=tf$ since $f$ is the left $\mathcal{T}_{Z}$-approximation of $j_{1}(Y)$. By the commutativity of the above diagram, we have $f=sh=stf$. Then $st$ is an automorphism since $f$ is left minimal (see [ASS, IV.Definition 1.1]) and there exists $l\in\opname{End}_{A}(K_{Y})$ such that $stl=id_{K_{Y}}$. By diagram chasing there exists $t': j_{1}(Y)\oplus P\longrightarrow T$ such that $g=t'i$. Observe the long exact sequence $(\divideontimes )$, we have finished to prove that $\Ext_{A}^{1}(K_{Y}, \mathcal{T}_{Z})=0$.

(ii) Finally, it is enough to show that $\mathcal {T}_{Z}\subset \opname{Fac}Z$. For any $M\in \mathcal {T}_{Z}$, since $i_{-1}$ is exact, by [FZ, Lemma 3.1(4)] we have a short exact sequence $0\rightarrow i_{0}i_{-1}(M)\buildrel {\sigma _{M}} \over\rightarrow M\buildrel {\varepsilon_{M}} \over\rightarrow j_{-1}j_{0}(M)\rightarrow 0$. It is easy to check that $i_{0}i_{-1}(M)$ and $j_{-1}j_{0}(M)$ are both in $\mathcal {T}_{Z}$. Since $M\in \mathcal {T}_{Z}$, we have $i_{-1}(M)\in \mathcal {T}_{X}=\opname{Fac}X$ and $j_{0}(M)\in \mathcal {T}_{Y}=\opname{Fac}Y$. Then there exist $X_{M}\in \opname{add}X$ and $Y_{M}\in \opname{add}Y$ such that $\pi_{1}: X_{M}\longrightarrow i_{-1}(M)$ and $\pi_{2}: Y_{M}\longrightarrow j_{0}(M)$ are epic. Since $i_{0}$ is exact, there is an epimorphism $h_{1}: i_{0}(X_{M})\longrightarrow i_{0}i_{-1}(M)$. Since $j_{-1}$ is exact and $j_{1}\longrightarrow j_{-1}\longrightarrow 0$, there is an epimorphism $h_{2}: j_{1}(Y_{M})\longrightarrow j_{-1}j_{0}(M)$.

Let $f': j_{1}(Y_{M})\longrightarrow \widetilde{Y_{M}}$ be the minimal left $\mathcal {T}_{Z}$-approximation of $j_{1}(Y_{M})$, then we have $\widetilde{Y_{M}}\in \opname{add}(K_{Y})$ since $Y_{M}\in \opname{add}Y$. Note that $h_{2}: j_{1}(Y_{M})\longrightarrow j_{-1}j_{0}(M)$ is epic and $j_{-1}j_{0}(M)\in \mathcal {T}_{Z}$, there exists an epimorphism $h_{3}: \widetilde{Y_{M}}\longrightarrow j_{-1}j_{0}(M)$ such that $h_{2}=h_{3}f'$. Consider the following diagram:
$$\xymatrix{
& 0\ar[r] & i_{0}(X_{M})\ar[r]^{i'}\ar[d]_{h_{1}} & i_{0}(X_{M})\oplus\widetilde{Y_{M}} \ar[r]^{\pi'}\ar[d]_{(\sigma _{M}h_{1}, m)} & \widetilde{Y_{M}}\ar[r]\ar[d]^{h_{3}}\ar@{.>}[dl]^{m} & 0&\\
& 0\ar[r] &i_{0}i_{-1}(M)\ar[r]_{\sigma _{M}} & M\ar[r]_{\varepsilon_{M}} & j_{-1}j_{0}(M)\ar[r] & 0.&}$$
Observe that $\widetilde{Y_{M}}$ is Ext-projective in $\mathcal {T}_{Z}$ since $\widetilde{Y_{M}}\in \opname{add}(K_{Y})$ and $K_{Y}$ is Ext-projective in $\mathcal {T}_{Z}$. Then there exists $m: \widetilde{Y_{M}}\longrightarrow M$ such that $h_{3}=\varepsilon_{M}\cdot m$ since $i_{0}i_{-1}(M)\in\mathcal {T}_{Z}$. Now we define $h'=(\sigma _{M}h_{1}, m): i_{0}(X_{M})\oplus\widetilde{Y_{M}}\longrightarrow M$. It is easy to check that the above diagram is commutative. By Snake Lemma, we know that $h'$ is epic since $h_{1}$ and $h_{3}$ are both epic. It follows that $M\in \opname{Fac}Z$ since $i_{0}(X_{M})\oplus\widetilde{Y_{M}}\in\opname{add}Z$.
\end{proof}
The following remark states what we will mean by \emph{gluing support $\tau$-tilting modules via symmetric ladders of height 2}.
\begin{remark}
Let $\mathcal {R}$ be a symmetric ladder of $\mod A$ relative to $\mod B$ and $\mod C$ of height $2$. We say that a support $\tau$-tilting module $Z\in\mod A$ is glued from support $\tau$-tilting modules $X\in\mod B$ and $Y\in\mod C$ with respect to $\mathcal {R}$ if $Z$ is obtained by the construction of Theorem 4.4, that is, $Z$ corresponds to the functorially finite torsion class glued from the functorially finite torsion classes associated to $X$ and $Y$ with respect to $\mathcal {R}$.
\end{remark}

We conclude this section with some examples illustrating our results. We show the process of gluing functorially finite torsion classes via the lower recollement (see Corollary 3.4 or Corollary 3.5). Then by Theorem 4.4 we show how to glue support $\tau$-tilting modules via the ladder. The process of gluing left finite semibrick via the upper recollement is easy by Proposition 4.1.

\begin{example}
Let $B$ be the path algebra of quiver $\cdot 1$ and $C$ the path algebra of quiver $\cdot 2$. Then the matrix algebra $A$ is the path algebra of quiver $1\longleftarrow 2$. They produce a symmetric ladder of height $2$ (see $(\mathcal {L})$ in Section 3).

In the following, we illustrate the construction of support $\tau$-tilting $A$-modules in Theorem 4.4.

\begin{table}[htbp]
\begin{minipage}{0.48\linewidth}
\centering
  \caption{\label{tab:test}}
 \begin{tabular}{lcl}
  \toprule
\quad\quad $\opname{f-tors\,}B  $\quad\quad & $\opname{f-tors\,}A$ \quad\quad & $\opname{f-tors\,}C$ \\
  \midrule
  \quad\quad $\mod B$ &
  $\mod A$ & \quad
  $\mod C$ \\\\
  \quad\quad $\mod B$&
  $\opname{add}(1)$\quad &
  \quad\quad
  $0$\\\\
  \quad\quad\quad $0$ &
  $\opname{add}(2)$
  \quad & \quad
  $\mod C$ \\\\
  \quad\quad\quad $0$ &
  $0$
  \quad & \quad\quad
  $0$ \\\\
   \bottomrule
 \end{tabular}
 \end{minipage}\begin{minipage}{0.48\linewidth}
 \centering
\caption{\label{tab:test}}
 \begin{tabular}{lcl}
  \toprule
$\opname{s\tau-tilt\,}B  $\quad\quad & $\opname{s\tau-tilt\,}A$ \quad\quad & $\opname{s\tau-tilt\,}C$ \\
  \midrule
  \quad
  $\color{red}{1}$ &
  $\begin{smallmatrix}\color{red}{1}\end{smallmatrix}\oplus
  \begin{smallmatrix}\color{red}{2}\\1\end{smallmatrix}$ & \quad\quad
  $\color{red}{2}$ \\\\
  \quad $\color{red}{1}$&
  $\color{red}{1}$\quad &
  \quad\quad
  $0$\\\\
  \quad $0$ &
  $\color{red}{2}$
  \quad & \quad\quad
  $\color{red}{2}$ \\\\
   \quad $0$ &
  $0$
  \quad & \quad\quad
  $0$ \\\\
   \bottomrule
\end{tabular}
\end{minipage}
\end{table}
In table 1, we list gluing of functorially finite torsion classes. In table 2, we list the construction of support $\tau$-tilting modules by gluing functorially finite torsion classes via the ladder, and the bricks in the corresponding left finite semibrick for each support $\tau$-tilting are read in color. Here we obtain 4 support $\tau$-tilting $A$-modules and the left one $\begin{smallmatrix}2\end{smallmatrix}\oplus
  \begin{smallmatrix}2\\1\end{smallmatrix}$ can not be obtained by this way.

Note that $j_{1}=?\ten_C e_2A$ preserves projective modules, $j_{1}(2)=\begin{smallmatrix}2\\1\end{smallmatrix}$. Observe the third rows in the both tables, $\begin{smallmatrix}2\\1\end{smallmatrix}$ is not in $\opname{add}(2)$, so we take the left $\opname{add}(2)$-approximation $\begin{smallmatrix}2\\1\end{smallmatrix}\longrightarrow 2$.
\end{example}

\begin{example}
Let $B$ be the path algebra of quiver $1\longrightarrow 2$ and $C$ the path algebra of quiver $\cdot 3$. Then the matrix algebra $A$ is the path algebra of quiver $3\longrightarrow 1\longrightarrow 2$. They produce a symmetric ladder of height $2$ (see $(\mathcal {L})$ in Section 3).

In the following, we illustrate the construction of support $\tau$-tilting $A$-modules in Theorem 4.4.
\begin{table}[htbp]
\begin{minipage}{0.48\linewidth}
\centering
  \caption{\label{tab:test}}
\resizebox{\textwidth}{55mm}{
 \begin{tabular}{lcl}
  \toprule
\quad\quad $\opname{f-tors\,}B  $\quad & $\opname{f-tors\,}A$ \quad\quad & $\opname{f-tors\,}C$ \\
  \midrule
  \quad\quad $\mod B$ &
  $\mod A$ &
  $\mod C$ \\\\
  \quad $\mod B/\opname{add}(2)$&
  $\mod A/\opname{add}(2)$\quad &
  $\mod C$\\\\
  \quad\quad $\opname{add}(1)$ &
  $\opname{add}(1\oplus {\begin{smallmatrix}3\\1\end{smallmatrix}}\oplus 3)$
  \quad &
  $\mod C$ \\\\
 \quad\quad $\opname{add}(2)$ &
  $\opname{add}(2\oplus 3)$
  \quad &
  $\mod C$ \\\\
  \quad\quad\quad $0$ &
  $\opname{add}(3)$
  \quad &
  $\mod C$ \\\\
  \quad\quad $\mod B$ &
  $\opname{add}(2\oplus {\begin{smallmatrix}1\\2\end{smallmatrix}}\oplus 1)$ & \quad
  $0$ \\\\
  \quad $\opname{add}({\begin{smallmatrix}1\\2\end{smallmatrix}}\oplus 1)$&
  $\opname{add}({\begin{smallmatrix}1\\2\end{smallmatrix}}\oplus 1)$\quad &
  \quad
  $0$\\\\
  \quad\quad $\opname{add}(1)$ &
  $\opname{add}(1)$
  \quad & \quad
  $0$ \\\\
 \quad\quad $\opname{add}(2)$ &
  $\opname{add}(2)$
  \quad & \quad
  $0$ \\\\
  \quad\quad\quad $0$ &
  $0$
  \quad & \quad
  $0$ \\\\
   \bottomrule
 \end{tabular}}
 \end{minipage}\begin{minipage}{0.44\linewidth}
\centering
\caption{\label{tab:test}}
\resizebox{\textwidth}{55mm}{
 \begin{tabular}{lcl}
  \toprule
$\opname{s\tau-tilt\,}B  $\quad\quad & $\opname{s\tau-tilt\,}A$ \quad\quad & $\opname{s\tau-tilt\,}C$ \\
  \midrule
  \quad $\begin{smallmatrix}\color{red}{1}\\2\end{smallmatrix}\oplus
  \begin{smallmatrix}\color{red}{2}\end{smallmatrix}$ &
  $\begin{smallmatrix}\color{red}{1}\\2\end{smallmatrix}\oplus
  \begin{smallmatrix}\color{red}{2}\end{smallmatrix}\oplus\begin{smallmatrix}\color{red}{3}\\1\\2\end{smallmatrix}$ & \quad\quad
  $\color{red}{3}$ \\\\
  \quad $\begin{smallmatrix}\color{red}{1}\\ \color{red}{2}\end{smallmatrix}\oplus
  \begin{smallmatrix}1\end{smallmatrix}$&
  $\begin{smallmatrix}\color{red}{1}\\ \color{red}{2}\end{smallmatrix}\oplus
  \begin{smallmatrix}1\end{smallmatrix}\oplus\begin{smallmatrix}\color{red}{3}\\1\\2\end{smallmatrix}$\quad &
  \quad\quad
  $\color{red}{3}$\\\\
  \quad $\color{red}{1}$ &
  $\begin{smallmatrix}\color{red}{1}\end{smallmatrix}\oplus
  \begin{smallmatrix}\color{red}{3}\\1\end{smallmatrix}$
  \quad & \quad\quad
  $\color{red}{3}$ \\\\
 \quad $\color{red}{2}$ &
  ${\color{red}{2}}\oplus{\color{red}{3}}$
  \quad & \quad\quad
  $\color{red}{3}$ \\\\
  \quad $0$ &
  $\color{red}{3}$
  \quad & \quad\quad
  $\color{red}{3}$ \\\\
  \quad $\begin{smallmatrix}\color{red}{1}\\2\end{smallmatrix}\oplus
  \begin{smallmatrix}\color{red}{2}\end{smallmatrix}$ &
  $\begin{smallmatrix}\color{red}{1}\\2\end{smallmatrix}\oplus
  \begin{smallmatrix}\color{red}{2}\end{smallmatrix}$ & \quad\quad
  $0$ \\\\
  \quad $\begin{smallmatrix}\color{red}{1}\\ \color{red}{2}\end{smallmatrix}\oplus
  \begin{smallmatrix}1\end{smallmatrix}$&
  $\begin{smallmatrix}\color{red}{1}\\ \color{red}{2}\end{smallmatrix}\oplus
  \begin{smallmatrix}1\end{smallmatrix}$\quad &
  \quad\quad
  $0$\\\\
  \quad $\color{red}{1}$ &
  $\color{red}{1}$
  \quad & \quad\quad
  $0$ \\\\
 \quad $\color{red}{2}$ &
  $\color{red}{2}$
  \quad & \quad\quad
  $0$ \\\\
  \quad $0$ &
  $0$
  \quad & \quad\quad
  $0$ \\\\
   \bottomrule
\end{tabular}}
\end{minipage}
\end{table}

In table 3, we list gluing of functorially finite torsion classes. In table 4, we list the construction of support $\tau$-tilting modules by gluing functorially finite torsion classes via the ladder, and the bricks in the corresponding left finite semibrick for each support $\tau$-tilting are read in color.
Here we obtain 10 support $\tau$-tilting $A$-modules and the left four $\begin{smallmatrix}3\end{smallmatrix}\oplus\begin{smallmatrix}2\end{smallmatrix}\oplus
\begin{smallmatrix}3\\1\\2\end{smallmatrix}, \begin{smallmatrix}3\end{smallmatrix}\oplus\begin{smallmatrix}3\\1\end{smallmatrix}\oplus
\begin{smallmatrix}3\\1\\2\end{smallmatrix}, \begin{smallmatrix}3\\1\end{smallmatrix}\oplus\begin{smallmatrix}1\end{smallmatrix}\oplus
\begin{smallmatrix}3\\1\\2\end{smallmatrix}, \begin{smallmatrix}3\end{smallmatrix}\oplus\begin{smallmatrix}3\\1\end{smallmatrix}$
can not be obtained by this way.

Note $j_{1}=?\ten_C e_2A$ preserves projective modules, $j_{1}(3)=\begin{smallmatrix}3\\1\\2\end{smallmatrix}$. Observe the third rows in the both tables, $\begin{smallmatrix}3\\1\\2\end{smallmatrix}$ is not in $\opname{add}(1\oplus {\begin{smallmatrix}3\\1\end{smallmatrix}}\oplus 3)$, so we take the left $\opname{add}(1\oplus {\begin{smallmatrix}3\\1\end{smallmatrix}}\oplus 3)$-approximation $\begin{smallmatrix}3\\1\\2\end{smallmatrix}\longrightarrow \begin{smallmatrix}3\\1\end{smallmatrix}$. Observe the fourth rows in the both tables, $\begin{smallmatrix}3\\1\\2\end{smallmatrix}$ is not in $\opname{add}(2\oplus 3)$, so we take the left $\opname{add}(2\oplus 3)$-approximation $\begin{smallmatrix}3\\1\\2\end{smallmatrix}\longrightarrow 3$. Observe the fifth rows in the both tables, $\begin{smallmatrix}3\\1\\2\end{smallmatrix}$ is not in $\opname{add}(3)$, so we take the left $\opname{add}(3)$-approximation $\begin{smallmatrix}3\\1\\2\end{smallmatrix}\longrightarrow 3$.

\end{example}

\section{Triangular matrix algebras}
\medskip

Throughout this section, we concentrate on triangular matrix algebras. Let $B$ and $C$ be finite-dimensional $k$ algebras and $M$ a finitely generated $C$-$B$-bimodule. $A=\left(\begin{array}{cc} B & 0\\ {}_CM_B & C\end{array}\right)$ is the triangular matrix algebra. The construction of support $\tau$-tilting modules over triangular matrix algebra in [PMH, Theorem 4.2(1)] is exactly what we mean, but they did not have an explicit formula. In this section, the first main result is to give an explicit construction of support $\tau$-tilting modules over triangular matrix algebras and generalize the result of [GH, Theorem 4.3].
\begin{theorem}
If $X$ and $Y$ are respectively support $\tau$-tilting modules in $\opname{mod}B$ and $\opname{mod}C$, then $(X, 0)\oplus (X_{Y}, Y)_{f}$ is a support $\tau$-tilting $A$-module, where $X_{Y}$ and $f$ are defined by the minimal left $\opname{Fac}X$-approximation $f: Y\otimes _{C}M\longrightarrow X_{Y}$.
\end{theorem}
\begin{proof}
It is well known that there is a symmetric ladder of $\mod A$ relative to $\mod B$ and $\mod C$ of height $2$ as follows (see ($\mathcal {L}$) in Section 3)
\[
\xymatrix@C=5pc{\mod B
\ar@<+1ex>[r]|{inc} \ar@<-3ex>[r] &\mod A \ar@<1ex>[l]|{?\ten_A Ae_1}
\ar@<-3ex>[l]
 \ar@<+1ex>[r]|{?\ten_A Ae_2} \ar@<-3ex>[r] & \mod
C.\ar@<1ex>[l]|{inc} \ar@<-3ex>[l]|{?\ten_C e_2A}}
\]
In this case, $i_{0}(X)=(X,0)$ and $j_{1}(Y)=(Y\otimes _{C}M, Y)$. Since $Y\in \opname{Fac}Y$ and $f: Y\otimes _{C}M\longrightarrow X_{Y}$ is the minimal left $\opname{Fac}X$-approximation of $Y\otimes _{C}M$, one can check that $(f, id): (Y\otimes _{C}M, Y)\longrightarrow (X_{Y}, Y)_{f}$ is the minimal left $\mathcal {T}_{Z}(=\{(R, S)_{g}|R\in\opname{Fac}X, S\in\opname{Fac}Y\})$-approximation of $(Y\otimes _{C}M, Y)$. By Theorem 2.2, Corollary 3.5 and Theorem 4.4, $(X, 0)\oplus (X_{Y}, Y)_{f}$ is a support $\tau$-tilting $A$-module.
\end{proof}

Actually the construction in Theorem 5.1 can be restricted to $\tau$-tilting modules. Before doing this, the following lemma in [PMH, Lemma 4.1(1)] is useful.
\begin{lemma}
If $\mathcal {T}_{X}$ and $\mathcal {T}_{Y}$ are respectively sincere functorially finite in $\mod B$ and $\mod C$, then $\mathcal {T}_{Z}=\{(X, Y)_{f}|X\in \mathcal {T}_{X}, Y\in \mathcal {T}_{Y}\}$ is sincere functorially finite in $\mod A$.
\end{lemma}

\begin{corollary}
Let $X$ and $Y$ be $\tau$-tilting modules in $\opname{mod}B$ and $\opname{mod}C$. Then we have
\begin{enumerate}
\item $(X, 0)\oplus (X_{Y}, Y)_{f}$ is a $\tau$-tilting $A$-module, where $X_{Y}$ and $f$ are defined by the minimal left $\opname{Fac}X$-approximation $f: Y\otimes _{C}M\longrightarrow X_{Y}$;
\item If $Y=C$, $(X, 0)\oplus (U, C)_{f}$ is exactly the Bongartz completion of $(X, 0)$, where $U$ and $f$ are defined by the minimal left $\opname{Fac}X$-approximation $f: C\otimes _{C}M\longrightarrow U$.
\end{enumerate}
\end{corollary}
\begin{proof}
(1) By Theorem 5.1, $(X, 0)\oplus (X_{Y}, Y)_{f}$ corresponds to the functorially finite torsion classes $\mathcal {T}_{Z}(=\{(R, S)_{g}|R\in\opname{Fac}X, S\in\opname{Fac}Y\})$. Since $X$ and $Y$ are $\tau$-tilting modules, by Theorem 2.2 $\opname{Fac}X$ and $\opname{Fac}Y$ are sincere functorially finite. By Lemma 5.2, we have $\mathcal {T}_{Z}$ is sincere functorially finite. Therefore $(X, 0)\oplus (X_{Y}, Y)_{f}$ is $\tau$-tilting by Theorem 2.2.

(2) By Theorem 5.1, $(X, 0)\oplus (U, C)_{f}$ corresponds to the functorially finite torsion classes $\mathcal {T}_{Z}(=\{(R, S)_{g}|R\in\opname{Fac}X, S\in\mod C\})$. Since $X$ is a $\tau$-tilting module, by [AIR, Theorem 2.12] we have $\opname{Fac}X=^{\bot }(\tau X)$. Thus $\mathcal {T}_{Z}=^{\bot }(\tau (X,0))$. It follows that $(X, 0)\oplus (U, C)_{f}$ is the Bongartz completion of $(X, 0)$ by definition of the Bongartz completion ([AIR, Theorem 2.10]).
\end{proof}

As a special case of Theorem 5.1, we obtain an independent proof of [GH, Theorem 4.3].
\begin{corollary}
If $X$ and $Y$ are respectively support $\tau$-tilting modules in $\opname{mod}B$ and $\opname{mod}C$ which satisfy $\Hom _{B}(Y\otimes _{C}M, \tau X)=0$ and $\Hom _{B}(eB, Y\otimes _{C}M)=0$, where $e$ is the maximal idempotent such that $\Hom_{B}(eB, X)=0$, then $(X, 0)\oplus (Y\otimes _{C}M, Y)$ is a support $\tau$-tilting $A$-module.
\end{corollary}
\begin{proof}
From [AIR, Corollary 2.13] we have $Y\otimes _{C}M\in ^{\bot}(\tau X)\cap (eB)^{\bot }=\opname{Fac}X$. So we omit the approximation in Theorem 5.1 and get the result easily.
\end{proof}

The following result gives us a reduction at the level of functorially finite torsion classes over triangular matrix algebras.

\begin{theorem}
If $X$ is a $\tau$-tilting $B$-module, we have an order-preserving bijection
$$
\{\mathcal {T}\in \opname{f-\tors} A|\opname{Fac}(X, 0)\subseteq \mathcal {T}\subseteq ^{\bot}(\tau(X, 0))\}\longleftrightarrow \opname{f-\tors} C
$$
given by $F: \mathcal {T}=\{(R, S)_{g}|R\in \opname{mod}B, S\in \opname{mod}C\} \mapsto \mathcal {G}=\{S|(R, S)_{g}\in \mathcal {T}\}$ with inverse $G: \mathcal {G} \mapsto \mathcal {T}=\{(R, S)_{g}|R\in \opname{Fac}X, S\in \mathcal {G}\}$.
\end{theorem}
\begin{proof}
(1) We first prove that
$\{\mathcal {T}\in \opname{f-\tors} A|\opname{Fac}(X, 0)\subseteq \mathcal {T}\subseteq ^{\bot}(\tau(X, 0))\}\\=\{\mathcal {T}|\mathcal {T}=\{(R, S)_{g}|R\in\opname{Fac}X, S\in\mathcal {G}\in\opname{f-\tors} C\}\}.$

It is easy to check that $\opname{Fac}(X, 0)=\{(R, S)_{g}|R\in \opname{Fac}X, S=0\}$ and $^{\bot}(\tau(X, 0))=\{(R, S)_{g}|R\in ^{\bot}(\tau X), S\in \opname{mod}C\}$. Since $X$ is $\tau$-tilting, we have $\opname{Fac}X=^{\bot}(\tau X)$. It follows that if $\opname{Fac}(X, 0)\subseteq \mathcal {T}\subseteq ^{\bot}(\tau(X, 0))$, then $\mathcal {T}=\{(R, S)_{g}|R\in \opname{Fac}X, S\in \mathcal {G}\subseteq \opname{mod}C\}$. We claim that if $\mathcal {T}\in \opname{f-\tors} A$, $\mathcal {G}\in \opname{f-\tors} C$.

(i) Since $\mathcal {T}$ is functorially finite, by Corollary 3.5, it is easy to see that $\mathcal {G}$ is functorially finite.

(ii) Now we check $\mathcal {G}\in \opname{\tors} C$. Let $0\longrightarrow P\longrightarrow Q\longrightarrow R\longrightarrow 0$ be an exact sequence in $\opname{mod}C$. Considering the exact sequence $0\longrightarrow 0\longrightarrow X\longrightarrow X\longrightarrow 0$ in $\opname{mod}B$, then we have an exact sequence $$0\longrightarrow (0, P) \longrightarrow(X, Q) \longrightarrow(X, R) \longrightarrow 0$$ in $\opname{mod}A$. If $Q\in \mathcal {G}$, $(X, Q)\in \mathcal {T}$. It follows that $(X, R)\in \mathcal {T}$ since $\mathcal {T}$ is a torsion class. Thus $R\in \mathcal {G}$. If $P, R\in \mathcal {G}$, $(0, P), (X, R)\in \mathcal {T}$. It follows that $(X, Q)\in \mathcal {T}$ since $\mathcal {T}$ is a torsion class. Thus $Q\in \mathcal {G}$.

Therefore $\{\mathcal {T}\in \opname{f-\tors} A|\opname{Fac}(X, 0)\subseteq \mathcal {T}\subseteq ^{\bot}(\tau(X, 0))\}\subseteq \{\mathcal {T}|\mathcal {T}= \{(R, S)_{g}|R\in\opname{Fac}X, S\in\mathcal {G}\in\opname{f-\tors} C\}\}.$ The opposite inclusion is clear by the gluing of functorially finite torsion classes.

It is an easy observation that the maps $F$ and $G$ are well defined and they are inverse of each other.

(2) Note that if $\mathcal {G}_{1}\subseteq \mathcal {G}_{2}$, $\mathcal {T}_{1}=G(\mathcal {G}_{1})\subseteq G(\mathcal {G}_{2})=\mathcal {T}_{2}$. Therefore the bijections are isomorphisms of partially ordered sets.
\end{proof}

Applying the above result, we are ready to give the second main result of this section. We study the subset of $\opname{s\tau-tilt\,}A$ given by $\opname{s\tau-tilt\,}_{(X, 0)}A:=\{(R, S)_{f}\in\opname{s\tau-tilt\,}A|(X, 0)\in\opname{add}(R, S)_{f}\}$.
\begin{theorem}
If $X$ is a $\tau$-tilting $B$-module, we have an order-preserving bijection
$$\mathrm{red}: \opname{s\tau-tilt\,}_{(X, 0)}A
\longleftrightarrow \opname{s\tau-tilt\,}C
$$
given by $Z=(R, S)_{f} \mapsto S$ with inverse $Y \mapsto (X, 0)\oplus (X_{Y}, Y)_{f}$ where $X_{Y}$ and $f$ are defined by the minimal left $\opname{Fac}X$-approximation $f: Y\otimes _{C}M\longrightarrow X_{Y}$. In particular, $\opname{s\tau-tilt\,}C$ can be embedded as an interval in $\opname{s\tau-tilt\,}A$.
\end{theorem}
\begin{proof}
By [AIR, Proposition 2.9], Theorem 2.2 and Theorem 5.5 we have a commutative diagram
$$\begin{xy}
(0,0)*+{\{\mathcal {T}\in \opname{f-\tors} A|\opname{Fac}(X, 0)\subseteq \mathcal {T}\subseteq ^{\bot}(\tau(X, 0))\}}="1",
(0,-20)*+{\opname{s\tau-tilt\,}_{(X, 0)}A}="2",
(60,0)*+{\opname{f-\tors} C}="3",
(60,-20)*+{\opname{s\tau-tilt\,}C}="4",
\ar"2";"1",\ar"1";"3",\ar"3";"4",\ar@{.>}"2";"4",
\end{xy}$$
in which each arrow is a bijection, and where the dashed arrow is given by $Z=(R, S)_{f} \mapsto P(F(\opname{Fac}Z))$ and the inverse is given by $Y \mapsto P(G(\opname{Fac}Y))$.

Hence to prove the theorem, we only need to show that
$P(F(\opname{Fac}Z))=S$ for any $Z=(R, S)_{f}\in\opname{s\tau-tilt\,}_{(X, 0)}A$ and
$P(G(\opname{Fac}Y))=(X, 0)\oplus (X_{Y}, Y)_{f}$ for any $Y\in \opname{s\tau-tilt\,}C$. By Theorem 5.5, we have $G(\opname{Fac}Y)=\{(R, S)_{g}|R\in \opname{Fac}X, S\in \opname{Fac}Y\}$. The corresponding support $\tau$-tilting $A$-module $P(G(\opname{Fac}Y))$ is $(X, 0)\oplus (X_{Y}, Y)_{f}$ by Theorem 5.1. Assume that $Z=(R, S)_{f}=(X, 0)\oplus (R_{1}, S)_{g}$, we have $\opname{Fac}Z=\opname{Fac}(X, 0)\oplus\opname{Fac}(R_{1}, S)_{g}=\{(U, 0)|U\in \opname{Fac}X\}\oplus \{(V, W)_{h}|V\in \opname{Fac}R_{1}, W\in\opname{Fac}S\}$.
By Theorem 5.5, we have $F(\opname{Fac}Z)=\opname{Fac}S$. It follows that $P(F(\opname{Fac}Z))=S$.
\end{proof}
\begin{corollary}
The bijection in Theorem 5.6 is compatible with mutation of support $\tau$-tilting modules.
\end{corollary}
\begin{proof}
By [AIR, Corollary 2.34], the exchange graph of $\opname{s\tau-tilt\,}A$ coincides with the Hasse diagram of $\opname{s\tau-tilt\,}A$. Since $\tau$-tilting reduction preserves the partial order in $\opname{f-\tors}$ (and hence in $\opname{s\tau-tilt\,}$) the claim follows.
\end{proof}
For any finite-dimensional algebra $A$, $Q(\opname{s\tau-tilt\,}A)$ is the support $\tau$-tilting quiver of $A$, i.e., the set of vertices in $Q(\opname{s\tau-tilt\,}A)$ is $\opname{s\tau-tilt\,}A$ and the arrows in $Q(\opname{s\tau-tilt\,}A)$ describe left mutations (see [AIR, Definition 2.29]). We give an easy example to illustrate Theorem 5.6.
\begin{example}
Let $B$ be the path algebra of quiver $1\longrightarrow 2$ and $C$ the path algebra of quiver $\cdot 3$. Then the triangular matrix algebra $A$ is the path algebra of quiver $3\longrightarrow 1\longrightarrow 2$.

Let $X$ be $\begin{smallmatrix}1\\ 2\end{smallmatrix} \oplus\begin{smallmatrix} 1\end{smallmatrix}$ a $\tau$-tilting $B$-module. The Bongartz completion of $(X, 0)$ is given by $\begin{smallmatrix}1\\2\end{smallmatrix}\oplus \begin{smallmatrix}1\end{smallmatrix} \oplus\begin{smallmatrix}3\\1\\2\end{smallmatrix}$. It is easy to see that $Q(\opname{s\tau-tilt\,}C)$ is given by the quiver $3\longrightarrow 0$. By Theorem 5.6, we have that $\opname{s\tau-tilt\,}C$ can be embedded as an interval in $\opname{s\tau-tilt\,}A$. In the following picture, we indicate this embedding in $Q(\opname{s\tau-tilt\,}A)$ by enclosing the image of $\opname{s\tau-tilt\,}C$ marked in red. Actually the colored support $\tau$-tilting $A$-modules are exactly the gluing of $X$ and support $\tau$-tilting $C$-modules.
$$\begin{xy} (0,0)*+{ \begin{smallmatrix}1\\2\end{smallmatrix}\oplus \begin{smallmatrix}2\end{smallmatrix}\oplus \begin{smallmatrix}3\\1\\2\end{smallmatrix}}="1",
(-25,-15)*+{\begin{smallmatrix}3\end{smallmatrix}\oplus \begin{smallmatrix}2\end{smallmatrix}\oplus \begin{smallmatrix}3\\1\\2\end{smallmatrix}}="2",
(0,-15)*+{\color{red}\begin{smallmatrix}1\\2\end{smallmatrix}\oplus \begin{smallmatrix}1\end{smallmatrix} \oplus\begin{smallmatrix}3\\1\\2\end{smallmatrix}}="3",
(25,-15)*+{
\begin{smallmatrix}1\\2\end{smallmatrix}\oplus \begin{smallmatrix}2\end{smallmatrix}}="4",
(-43,-30)*+{
\begin{smallmatrix}3\end{smallmatrix}\oplus \begin{smallmatrix}2\end{smallmatrix}}="5",
(-25,-30)*+{
\begin{smallmatrix}3\end{smallmatrix}\oplus \begin{smallmatrix}3\\1\end{smallmatrix}\oplus \begin{smallmatrix}3\\1\\2\end{smallmatrix}}="6",
(0,-30)*+{
\begin{smallmatrix}3\\1\end{smallmatrix}\oplus \begin{smallmatrix}1\end{smallmatrix}\oplus \begin{smallmatrix}3\\1\\2\end{smallmatrix}}="7",
(25,-30)*+{\color{red}
\begin{smallmatrix}1\\2\end{smallmatrix}\oplus \begin{smallmatrix}1\end{smallmatrix}}="8",
(43,-30)*+{\begin{smallmatrix}2\end{smallmatrix}}="9",
(-25,-45)*+{
\begin{smallmatrix}3\end{smallmatrix}\oplus \begin{smallmatrix}3\\1\end{smallmatrix}}="10",
(0,-45)*+{
\begin{smallmatrix}3\\1\end{smallmatrix} \oplus\begin{smallmatrix}1\end{smallmatrix}}="11",
(25,-45)*+{ \begin{smallmatrix}1\end{smallmatrix}}="12",
(-25,-60)*+{ \begin{smallmatrix}3\end{smallmatrix}}="13",
(25,-60)*+{
\begin{smallmatrix}0\end{smallmatrix}}="14",
 \ar"1";"2",
\ar"1";"3", \ar"1";"4", \ar"2";"5", \ar"2";"6", \ar"3";"7",
\ar"3";"8", \ar"4";"8", \ar"4";"9", \ar"5";"13", \ar@/_1.7pc/"5";"9",
\ar"6";"10", \ar"7";"6", \ar"7";"11", \ar"8";"12",
\ar"9";"14",\ar"10";"13", \ar"11";"10", \ar"11";"12",
\ar"12";"14", \ar"13";"14",
\end{xy}$$

\end{example}

Now we compare our reduction (Theorem 5.6) with Jasso's reduction [J, Theorem 3.16]. Let $X$ be a $\tau$-tilting $B$-module. By Corollary 5.3(2), $(X, 0)$ is a $\tau$-rigid $A$-module and $T=(X, 0)\oplus (U, C)_{f}$ is the Bongartz completion of $(X, 0)$. Set
$$B'=\End_{A}T \quad\text{and}\quad C'=B'/\langle e\rangle ,$$
where $e$ is the idempotent corresponding to the projective $B'$-module $\Hom_{A}(T, (X, 0))$. By Jasso's reduction [J, Theorem 3.16], there is an order-preserving bijection $$\mathrm{red}: \opname{s\tau-tilt\,}_{(X, 0)}A
\longleftrightarrow \opname{s\tau-tilt\,}C'.$$

It is natural to ask if there is any relation between algebras $C$ and $C'$. In the following, we show that $C$ coincides with $C'$ in Jasso's reduction.
\begin{theorem}
With the above setting, $C$ is isomorphic to $C'$.
\end{theorem}
\begin{proof}
By the constructions of $B'$ and $C'$, we know that

$C'=\End_{A}(U, C)_{f}/\{$ endomorphisms of $(U, C)_{f}$ factoring through $\opname{add} (X, 0)\}.$

The endmorphisms of $(U, C)_{f}$ are  pairs of $(\alpha, \beta)$ such that the following diagram
$$\xymatrix{
& C\otimes _{C}M\ar[rr]^{f}\ar[d]_{\beta \otimes id_M} && U\ar[d]^{\alpha}&\\
& C\otimes _{C}M\ar[rr]^{f} && U&}$$
commutes, where $\alpha\in \End_{B}U$ and $\beta\in \End_{C}C$.

The pair $(\alpha, \beta)$ factors through $\opname{add} (X, 0)$ if and only if there exists $\gamma\in\Hom_{B}(U, X^{\oplus n})$ and $\delta\in\Hom_{B}(X^{\oplus n}, U)$ such that $(\alpha, \beta)=(\delta, 0)(\gamma, 0)$, i.e., the following diagram
$$\xymatrix{
& C\otimes _{C}M\ar[rr]^{f}\ar[d]_{0} && U\ar[d]^{\gamma}&\\
& 0\ar[rr]^{0}\ar[d]_{0} && X^{\oplus n} \ar[d]^{\delta}&\\
& C\otimes _{C}M\ar[rr]^{f} && U&}$$
commutes, where $\delta\gamma=\alpha$ and $\beta=0$.

Thus $C'=\End_{A}(U, C)_{f}/\{(\alpha, 0)|$ $\alpha$ factoring through $\opname{add}X\}.$ Now we define an algebra homomorphism $\theta: \End_{A}(U, C)_{f}\longrightarrow \End_{C}C\cong C$ via $(\alpha, \beta)\longmapsto \beta$.

We claim that $\theta$ is surjective. Let $\beta\in\End_{C}C$. Since $f$ is the left $\opname{Fac} X$-approximation, there exists $\alpha\in\End_{B}U$ such the following diagram
$$\xymatrix{
& C\otimes _{C}M\ar[rr]^{f}\ar[d]_{\beta \otimes id_M} && U\ar@{.>}[d]^{\alpha}&\\
& C\otimes _{C}M\ar[rr]^{f} && U&}$$
commutes. It follows that $(\alpha, \beta)\in\End_{A}(U, C)_{f}$ and $\theta((\alpha, \beta))=\beta$.

Observe that $\theta$ yields a surjective algebra homomorphism $\varphi: C'\longrightarrow C$. We only have to prove that $\varphi$ is injective.\\
(i) Assume that $(\alpha, \beta)\in \opname{ker}\varphi$. Then $\beta=0$. It follows that $\alpha f=0$ and hence $\alpha$ factors through $\opname{coker}f$.\\
(ii) $\opname{coker}f\in \opname{add}X$. By the dual of Wakamatsu's Lemma [AR, Lemma 1.3], we have $\Ext^{1}_{B}(\opname{coker}f, \opname{Fac}X)=0$. Thus $\opname{coker}f$ is an $\Ext$-projective object in $\opname{Fac}X$ and hence belongs to $\opname{add}X$ since $X$ is $\tau$-tilting.

We have finished to prove that $\varphi$ is an isomorphism.
\end{proof}

As another application of Theorem 5.5, we give the third main result of this section. We denote by $l(\alpha)$ the length of a maximal green sequence $\alpha$. By [KD, Theorem 5.3], for a finite-dimensional algebra $A$, torsion classes in a maximal green sequence are always functorially finite and the length of a maximal green sequence is exactly the length of the path starting at 0 in the Hasse quiver of functorially torsion classes of $\mod A$.
\begin{theorem}
If $\alpha: 0=\mathscr{X}_{0}\subsetneq \mathscr{X}_{1}\subsetneq \cdot \cdot \cdot \subsetneq \mathscr{X}_{r-1}\subsetneq \mathscr{X}_{r}=\mod B$ and $\beta: 0=\mathscr{Y}_{0}\subsetneq \mathscr{Y}_{1}\subsetneq \cdot \cdot \cdot \subsetneq \mathscr{Y}_{s-1}\subsetneq \mathscr{Y}_{s}=\mod C$ are respectively maximal green sequences in $\mod B$ and $\mod C$, then $\gamma: 0=\mathscr{Z}_{0}\subsetneq \mathscr{Z}_{1}\subsetneq \cdot \cdot \cdot \mathscr{Z}_{r}\subsetneq \mathscr{Z}_{r+1} \cdot \cdot \cdot\subsetneq \mathscr{Z}_{r+s-1}\subsetneq \mathscr{Z}_{r+s}=\mod A$ is a maximal green sequence in $\mod A$, where $\mathscr{Z}_{i}=\{(X,0)|X\in \mathscr{X}_{i}\}$ for $0\leq i\leq r$ and $\mathscr{Z}_{r+j}=\{(X,Y)_{f}|X\in \mod B, Y\in\mathscr{Y}_{j} \}$ for $1\leq j\leq s$. In particular, $l(\gamma)=l(\alpha)+l(\beta)$.
\end{theorem}
\begin{proof}
Since $\alpha$ is a maximal green sequence, $\mathscr{X}_{i}\subsetneq \mathscr{X}_{i+1}$ are all covering relations for $0\leq i\leq r-1$. It follows that $\mathscr{Z}_{i}\subsetneq \mathscr{Z}_{i+1}$ are covering relations for $0\leq i\leq r-1$. Since $\beta$ is a maximal green sequence, $\mathscr{Y}_{j}\subsetneq \mathscr{Y}_{j+1}$ are all covering relations for $0\leq j\leq s-1$. Let $X=B$ in Theorem 5.5, observe that the map $G$ is order-preserving, it follows that $\mathscr{Z}_{r+j}\subsetneq \mathscr{Z}_{r+j+1}$ are covering relations for $0\leq j\leq s-1$. Therefore $\mathscr{Z}_{i}\subsetneq \mathscr{Z}_{i+1}$ are all covering relations for $0\leq i\leq r+s-1$ and $\gamma$ is a maximal green sequence.
\end{proof}
By [DIRRT, Theorem 3.3], there is a unique brick associated with each covering relation of torsion classes which called brick labelling in [DIRRT, Definition 3.5]. Consequently, any maximal green sequence gives rise to a sequence of isomorphism classes of bricks. Keller and Demonet in [KD, Theorem 5.4] showed that maximal green sequences for finite dimensional algbras are parameterized by such sequences called complete forward hom-orthogonal sequences. A complete forward hom-orthogonal sequence for $A$ first mentioned by Igusa in [I, Definition 2.2] is a sequence of bricks $M_{1}, M_{2}, \cdot\cdot\cdot, M_{k}$ of $A$ where $\opname{Hom}_{A}(M_{i}, M_{j})=0$ for all $i\textless j $ that cannot be refined in such a way that preserve this condition. Here we also give a statement of Theorem 5.10 in language of bricks.
\begin{corollary}
If $(M_{1}, M_{2}, \cdot\cdot\cdot, M_{r})$ and $(N_{1}, N_{2}, \cdot\cdot\cdot, N_{s})$ are respectively complete forward hom-orthogonal sequences in $\mod B$ and $\mod C$, then $(M_{1}, M_{2}, \cdot\cdot\cdot, M_{r}, N_{1}, N_{2}, \cdot\cdot\cdot, N_{s})$ is a complete forward hom-orthogonal sequence in $\mod A$, where $M_{i}\cong (M_{i}, 0)$ can be regarded as an $A$-module for any $1\leq i\leq r$ and $N_{j}\cong (0, N_{j})$ can be regarded as an $A$-module for any $1\leq j\leq s$.
\end{corollary}
\begin{proof}
Assume that $(M_{1}, M_{2}, \cdot\cdot\cdot, M_{r})$ and $(N_{1}, N_{2}, \cdot\cdot\cdot, N_{s})$ respectively arise from maximal green sequences $0=\mathscr{X}_{0}\subsetneq \mathscr{X}_{1}\subsetneq \cdot \cdot \cdot \subsetneq \mathscr{X}_{r-1}\subsetneq \mathscr{X}_{r}=\mod B$ and $0=\mathscr{Y}_{0}\subsetneq \mathscr{Y}_{1}\subsetneq \cdot \cdot \cdot \subsetneq \mathscr{Y}_{s-1}\subsetneq \mathscr{Y}_{s}=\mod C$. Then by Theorem 5.10, $0=\mathscr{Z}_{0}\subsetneq \mathscr{Z}_{1}\subsetneq \cdot \cdot \cdot \mathscr{Z}_{r}\subsetneq \mathscr{Z}_{r+1} \cdot \cdot \cdot\subsetneq \mathscr{Z}_{r+s-1}\subsetneq \mathscr{Z}_{r+s}=\mod A$ is a maximal green sequence in $\mod A$, where $\mathscr{Z}_{i}=\{(X,0)|X\in \mathscr{X}_{i}\}$ for $0\leq i\leq r$ and $\mathscr{Z}_{r+j}=\{(X,Y)_{f}|X\in \mod B, Y\in\mathscr{Y}_{j} \}$ for $1\leq j\leq s$, and the corresponding sequence of bricks $(M_{1}, M_{2}, \cdot\cdot\cdot, M_{r}, N_{1}, N_{2}, \cdot\cdot\cdot, N_{s})$ is a complete forward hom-orthogonal sequence.
\end{proof}
By Proposition 2.7, maximal green sequences induce maximal paths in the Hasse quiver of support $\tau$-tilting poset. Actually Corollary 5.11 means that if we want to construct maximal paths in the support $\tau$-tilting quiver of $A$ from that of $B$ and $C$, starting from support $\tau$-tilting module 0 we should label $\opname{brick} B$ on the arrows of the maximal paths in $Q(\opname{s\tau-tilt\,}B)$ first then $\opname{brick} C$ on the arrows of the maximal paths in $Q(\opname{s\tau-tilt\,}C)$. We give examples to illustrate our result.
\begin{example}
Let $B$ be the path algebra of quiver $\cdot 1$ and $C$ the path algebra of quiver $\cdot 2$. Then the matrix algebra $A$ is the path algebra of quiver $1\longleftarrow 2$. In the middle line of Table 5, we construct a maximal path $(\begin{smallmatrix} 1\end{smallmatrix}\oplus\begin{smallmatrix}2\\ 1\end{smallmatrix}\longrightarrow 1\longrightarrow 0)$ in $Q(\opname{s\tau-tilt\,}A)$. The corresponding maximal green sequence and complete forward hom-orthogonal sequence are respectively $0\subsetneq \opname{add}(1)\subsetneq \mod A$ and $(1, 2)$.
\begin{table}[htbp]
  \caption{\label{tab:test}}
 \begin{tabular}{lcl}
  \toprule
 $ Q(\opname{s\tau-tilt\,}B) $\quad\quad\quad & $Q(\opname{s\tau-tilt\,}A)$ \quad\quad\quad & $Q(\opname{s\tau-tilt\,}C)$ \\
  \midrule
  \quad\quad $\begin{xy}
(0,0)*+{\begin{smallmatrix} 1\end{smallmatrix}}="1",
(0,-15)*+{\begin{smallmatrix} 0\end{smallmatrix}}="2",
\ar^{\color{red}1}"1";"2",
\end{xy}$ &
$\begin{xy}
(0,0)*+{\begin{smallmatrix} 2\end{smallmatrix}\oplus\begin{smallmatrix}2\\ 1\end{smallmatrix}}="1",
(0,-15)*+{\begin{smallmatrix} 2\end{smallmatrix}}="2",
(18,0)*+{\begin{smallmatrix} 1\end{smallmatrix}\oplus\begin{smallmatrix}2\\ 1\end{smallmatrix}}="3",
(36,0)*+{\begin{smallmatrix} 1\end{smallmatrix}}="4",
(36,-15)*+{\begin{smallmatrix} 0\end{smallmatrix}}="5",
\ar"1";"2",\ar"3";"1",\ar^{\color{red}2}"3";"4",\ar"2";"5",
\ar"4";"5"^{\color{red}1}
\end{xy}$ \quad\quad
   & \quad\quad
  $\begin{xy}
(0,0)*+{\begin{smallmatrix} 2\end{smallmatrix}}="1",
(0,-15)*+{\begin{smallmatrix} 0\end{smallmatrix}}="2",
\ar^{\color{red}2}"1";"2",
\end{xy}$ \\\\
  \bottomrule
 \end{tabular}
\end{table}

\end{example}

\begin{example}
Let $B$ be the path algebra of quiver $1\longrightarrow 2$ and $C$ the path algebra of quiver $\cdot 3$. Then the matrix algebra $A$ is the path algebra of quiver $3\longrightarrow 1\longrightarrow 2$. In the middle line of Table 6, we construct 2 maximal paths in $Q(\opname{s\tau-tilt\,}A)$.

One is
$$\begin{smallmatrix} 1\\2\end{smallmatrix}\oplus\begin{smallmatrix}2\end{smallmatrix}\oplus
\begin{smallmatrix}3\\1\\2\end{smallmatrix}\longrightarrow\begin{smallmatrix} 1\\2\end{smallmatrix}\oplus\begin{smallmatrix}2\end{smallmatrix}\longrightarrow\begin{smallmatrix} 1\\2\end{smallmatrix}\oplus\begin{smallmatrix}1\end{smallmatrix}\longrightarrow \begin{smallmatrix}1\end{smallmatrix}\longrightarrow 0,$$
the corresponding complete forward hom-orthogonal sequence and maximal green sequence are respectively $(1, \begin{smallmatrix} 1\\2\end{smallmatrix}, 2, 3)$ and
$$0\subsetneq \opname{add}(1)\subsetneq \opname{add}(\begin{smallmatrix} 1\\2\end{smallmatrix}\oplus 1)\subsetneq \opname{add}(2\oplus\begin{smallmatrix} 1\\2\end{smallmatrix}\oplus 1)\subsetneq \mod A.$$

The other is
$$\begin{smallmatrix} 1\\2\end{smallmatrix}\oplus\begin{smallmatrix}2\end{smallmatrix}\oplus
\begin{smallmatrix}3\\1\\2\end{smallmatrix}\longrightarrow\begin{smallmatrix} 1\\2\end{smallmatrix}\oplus\begin{smallmatrix}2\end{smallmatrix}\longrightarrow \begin{smallmatrix}2\end{smallmatrix}\longrightarrow 0,$$
the corresponding complete forward hom-orthogonal sequence and maximal green sequence are respectively $(2, 1, 3)$ and
$$0\subsetneq \opname{add}(2)\subsetneq \opname{add}(2\oplus\begin{smallmatrix} 1\\2\end{smallmatrix}\oplus 1)\subsetneq \mod A.$$
\begin{table}[htbp]
  \caption{\label{tab:test}}
 \begin{tabular}{lcl}
  \toprule
 $ Q(\opname{s\tau-tilt\,}B) $\quad\quad\quad & $Q(\opname{s\tau-tilt\,}A)$ \quad\quad\quad & $Q(\opname{s\tau-tilt\,}C)$ \\
  \midrule
   $\begin{xy}
(0,0)*+{\begin{smallmatrix}1\\ 2\end{smallmatrix}\oplus \begin{smallmatrix} 2\end{smallmatrix}}="1",
(0,-15)*+{\begin{smallmatrix} 0\end{smallmatrix}}="2",
(18,0)*+{\begin{smallmatrix}1\\ 2\end{smallmatrix} \oplus\begin{smallmatrix} 1\end{smallmatrix}}="3",
(18,-15)*+{\begin{smallmatrix} 1\end{smallmatrix}}="4",
\ar^{\color{red}2}"1";"3",\ar^{\color{red}\begin{smallmatrix}1\\ 2\end{smallmatrix}}"3";"4",
\ar^{\color{red}1}"4";"2"
\end{xy}$ &
$\begin{xy}
(0,0)*+{\begin{smallmatrix}1\\ 2\end{smallmatrix}\oplus\begin{smallmatrix} 2\end{smallmatrix} \oplus\begin{smallmatrix}3\\ 1\\2\end{smallmatrix}}="1",
(0,-20)*+{\begin{smallmatrix} 0\end{smallmatrix}}="2",
(25,0)*+{\begin{smallmatrix}1\\ 2\end{smallmatrix}\oplus\begin{smallmatrix} 2\end{smallmatrix}}="3",
(50,0)*+{\begin{smallmatrix}1\\ 2\end{smallmatrix}\oplus\begin{smallmatrix} 1\end{smallmatrix}}="4",
(50,-20)*+{\begin{smallmatrix} 1\end{smallmatrix}}="5",
\ar^{\color{red}3}"1";"3", \ar^{\color{red}2}"3";"4",\ar^{\color{red}1}"5";"2",
\ar^{\color{red}\begin{smallmatrix}1\\ 2\end{smallmatrix}}"4";"5"
\end{xy}$ \quad\quad
   & \quad\quad
  $\begin{xy}
(0,0)*+{\begin{smallmatrix} 3\end{smallmatrix}}="1",
(0,-15)*+{\begin{smallmatrix} 0\end{smallmatrix}}="2",
\ar^{\color{red}3}"1";"2",
\end{xy}$ \\\\
\midrule
$\begin{xy}
(0,0)*+{\begin{smallmatrix}1\\ 2\end{smallmatrix} \oplus\begin{smallmatrix} 2\end{smallmatrix}}="1",
(15,0)*+{\begin{smallmatrix} 2\end{smallmatrix}}="2",
(15,-15)*+{\begin{smallmatrix} 0\end{smallmatrix}}="3",
\ar^{\color{red}1}"1";"2",
\ar^{\color{red}2}"2";"3"
\end{xy}$ &
$\begin{xy}
(0,0)*+{\begin{smallmatrix}1\\ 2\end{smallmatrix}\oplus \begin{smallmatrix} 2\end{smallmatrix}\oplus\begin{smallmatrix}3\\1\\ 2\end{smallmatrix}}="1",
(0,-15)*+{\begin{smallmatrix} 0\end{smallmatrix}}="2",
(40,0)*+{\begin{smallmatrix}1\\ 2\end{smallmatrix}\oplus \begin{smallmatrix} 2\end{smallmatrix}}="3",
(40,-15)*+{\begin{smallmatrix} 2\end{smallmatrix}}="4",
\ar^{\color{red}3}"1";"3",\ar^{\color{red}1}"3";"4",
\ar^{\color{red}2}"4";"2"
\end{xy}$ \quad\quad
   & \quad\quad
  $\begin{xy}
(0,0)*+{\begin{smallmatrix} 3\end{smallmatrix}}="1",
(0,-15)*+{\begin{smallmatrix} 0\end{smallmatrix}}="2",
\ar^{\color{red}3}"1";"2",
\end{xy}$ \\\\
  \bottomrule
 \end{tabular}
\end{table}\\
\end{example}

\bigskip
{\bf Acknowledgement.} This work was supported by the NSFC (Grant No. 12201211) and the China Scholarship Council (CSC No. 202109710002). The author is grateful to Prof. Dong Yang for his many helpful discussions. She also would like to thank Prof. Xiaowu Chen and Prof. Nan Gao for their several discussions. She thanks Prof. Guodong Zhou for pointing out reference [FZ]. Thanks also go to Prof. Osamu Iyama for pointing out to the author that $\varphi$ in Theorem 5.9 is injective and Dr. Peigen Cao for suggesting maximal green sequences. She thanks the support and excellent working conditions at the University of Tokyo. Moreover, She thanks anonymous referees for useful comments.

\end{document}